\documentclass[11pt]{article}
\usepackage{latexsym,amsfonts,amssymb,amsmath,amsthm}
\usepackage{graphicx}
\usepackage{subcaption}
\usepackage[usenames,dvipsnames]{color}
\usepackage{ulem}
\usepackage{caption}
\usepackage{float}
\usepackage{xcolor}
\usepackage{bbm}
\usepackage{comment}
\usepackage{appendix}

\begin{document}

\newtheorem{tm}{Theorem}[section]
\newtheorem{prop}[tm]{Proposition}
\newtheorem{defin}{Definition}[section]
\newtheorem{coro}{Corollary}[section]
\newtheorem{lem}{Lemma}[section]
\newtheorem{assumption}{Assumption}[section]
\newtheorem{rk}{Remark}[section]
\newtheorem{nota}{Notation}[section]
\numberwithin{equation}{section}

\newcommand{\mg}{\color{RedViolet}}

    \newcommand{\lb}{\label}
  \newcommand{\beq}{\begin{equation}}
    \newcommand{\eeq}{\end{equation}}

\newcommand{\stk}[2]{\stackrel{#1}{#2}}
\newcommand{\dwn}[1]{{\scriptstyle #1}\downarrow}
\newcommand{\upa}[1]{{\scriptstyle #1}\uparrow}
\newcommand{\nea}[1]{{\scriptstyle #1}\nearrow}
\newcommand{\sea}[1]{\searrow {\scriptstyle #1}}
\newcommand{\csti}[3]{(#1+1) (#2)^{1/ (#1+1)} (#1)^{- #1
 / (#1+1)} (#3)^{ #1 / (#1 +1)}}
\newcommand{\RR}[1]{\mathbb{#1}}

\newcommand{\rd}{{\mathbb R^d}}
\newcommand{\ep}{\varepsilon}
\newcommand{\rr}{{\mathbb R}}
\newcommand{\alert}[1]{\fbox{#1}}
\newcommand{\eqd}{\sim}

\def\un{\underline}
\def\ba{\overline}

\def\p{\partial}
\def\R{{\mathbb R}}
\def\N{{\mathbb N}}
\def\Q{{\mathbb Q}}
\def\C{{\mathbb C}}
\def\l{\left(}
\def\r{ \right) }
\def\t{\tau}
\def\k{\kappa}
\def\a{\alpha}
\def\la{\lambda}
\def\De{\Delta}
\def\de{\delta}
\def\ga{\gamma}
\def\Ga{\Gamma}
\def\ep{\varepsilon}
\def\eps{\varepsilon}
\def\si{\sigma}
\def\Re {{\rm Re}\,}
\def\Im {{\rm Im}\,}
\def\E{{\mathbb E}}
\def\P{{\mathbb P}}
\def\Z{{\mathbb Z}}
\def\D{{\mathbb D}}
\def\calS{\mathcal{S}}
\def\O{{\Omega_T}}
\def\calC{\mathcal{C}}
\def\calL{\mathcal{L}}
\def\calX{\mathcal{X}}

\newcommand{\ceil}[1]{\lceil{#1}\rceil}

\newcommand{\bbR}{{\mathbb R}}

\newcommand{\calB}{{\mathcal B}}

\title{Global existence of weak solutions to chemotaxis models with porous medium diffusion,
linear production, and logistic source on 
$\mathbb{R}^N$}

\author{
Zulaihat Hassan  \\
Department of Mathematics \\
Michigan State University, MI 48823, U.S.A. }

\date{}
\maketitle

\begin{abstract}
 This paper investigates the global solvability, boundedness, and uniqueness of weak solutions to the chemotaxis system
\begin{equation*}
\begin{cases}
u_t = \Delta u^m - \chi \nabla \cdot (u \nabla v) + u(a - b u), & \text{in } (0,\infty)\times\mathbb{R}^N, \\
\tau v_t = \Delta v - \lambda v + \mu u, & \text{in } (0,\infty)\times\mathbb{R}^N,
\end{cases}
\end{equation*}
where \(m>1\), \(\tau\in\{0,1\}\), \(\lambda,\mu,a,b>0\), and \(\chi\in\mathbb{R}\). For every \(m>1\), we establish the existence of weak solutions for initial data that are not necessarily integrable, although such solutions need not be bounded in general. We then show that globally bounded weak solutions exist in the parabolic-parabolic case \((\tau=1)\) when \(m>\frac{2N}{N+2}\), and also when \(1<m\le \frac{2N}{N+2}\) provided that the logistic damping coefficient \(b\) is sufficiently large. In the parabolic-elliptic case \((\tau=0)\), we prove the existence of globally bounded weak solutions when \(m>2-\frac{2}{N}\), and also when \(1<m\le 2-\frac{2}{N}\) provided that \(b\) is sufficiently large. Finally, for \(1<m\le 3\), we prove uniqueness of weak solutions that are H\"older continuous up to the initial time.
\end{abstract}

\medskip


\medskip


\tableofcontents
\section{Introduction}

\subsection{Overview}

In this paper, we consider a chemotaxis system with nonlinear diffusion of porous-medium type, motivated by enhanced cell stress in densely packed regions, together with a logistic source term accounting for cell proliferation and death, and linear production of the chemical. The system is posed on the whole space \(\mathbb R^N\) and is given by
\begin{equation}
\label{main-eq}
\begin{cases}
u_t = \Delta u^m - \chi \nabla \cdot (u \nabla v) + u(a - b u),\quad & x \in \mathbb{R}^N, \,\, t>0, \\
\tau v_t = \Delta v - \lambda v + \mu u, \quad & x \in \mathbb{R}^N,\,\,  t>0,\\
u(0,x) = u_0(x),\,\, \tau v(0, x) = \tau v_0(x), \quad & x\in \R^N,
\end{cases}
\end{equation}
where \(m>1\), \(a,b,\lambda,\mu>0\), \(\tau \in \{0,1\}\), and \(\chi\in \mathbb R\). We consider initial data that are not necessarily integrable and satisfy
\begin{equation}\label{initial-cond-eq}
u_0(\cdot)\in L^\infty(\R^N),\quad \tau v_0(\cdot)\in W^{1,\infty}(\R^N).
\end{equation}
The variable \(u\) represents the cell density, while \(v\) denotes the concentration of the chemical signal. In \eqref{main-eq}, the case \(\tau=0\) corresponds to the parabolic-elliptic system, which models a situation in which the chemical substance diffuses very fast, whereas \(\tau=1\) corresponds to the fully parabolic case. The constants \(\lambda\) and \(\mu\) represent the degradation and production rates of the chemical substance, respectively. System \eqref{main-eq} can be viewed as a special case of the general Keller--Segel model introduced by Keller and Segel in the 1970s \cite{keller1970initiation,keller1971model} to describe chemotactic movement of a single species in response to a single chemical signal. Chemotaxis, namely the directed motion of organisms toward or away from chemical cues in their environment, is observed in many microorganisms, including {\it Escherichia coli}, and is involved in important biological processes such as aggregation, pattern formation, and population stabilization.

While global existence and uniqueness of weak solutions to \eqref{main-eq} have been studied under integrability assumptions on the initial data, we establish global existence, boundedness, as well as uniqueness in a suitable H\"older class, without requiring the initial data to be integrable. Studying chemotaxis models on \(\mathbb R^N\) with non-integrable initial data is important both mathematically and biologically. In many realistic situations, the initial population or chemical distribution may not have finite total mass, for instance when the species is spread over a large habitat, approaches a nonzero background state at infinity, or exhibits front-like profiles. Such settings cannot be described within the framework of integrable initial data. From the mathematical point of view, allowing non-integrable data significantly enlarges the class of admissible solutions and makes it possible to investigate large-time behavior, spreading phenomena, and related dynamical properties in spatially unbounded environments; see \cite{griette2023speed,hassan2025spreading,henderson2024traveling,ji2021reducing,salako2016spreading,salako2018existence,shen2021spreading}.

\smallskip

Consider the counterpart of \eqref{main-eq} on a smooth bounded domain with homogeneous Neumann boundary conditions:
\begin{equation}
\label{special-eq2}
\begin{cases}
u_t=\Delta u^m-\chi\nabla \cdot(u\nabla v)+u(a-bu),\quad &x\in\Omega, t>0,\\[1mm]
\tau v_t=\Delta v+\mu u-\lambda v,\quad &x\in\Omega, t>0,\\[1mm]
\frac{\partial u}{\partial \nu}=\frac{\partial v}{\partial \nu} =0, &x\in \partial \Omega, t>0, \\[1mm]
u(0,x)=u_0(x),\quad \tau v(0,x)=\tau v_0(x),\quad &x\in\Omega.
\end{cases}
\end{equation}
Systems of the form \eqref{special-eq2} have been widely studied in both the parabolic-elliptic case \((\tau=0)\) and the parabolic-parabolic case \((\tau=1)\); see, for instance,
\cite{cao2014boundedness,cieslak2008finite,li2015global,tao2012boundedness,wang2014quasilinear,winkler2010does,yang2015boundedness,zhang2015boundedness}.
For the parabolic--elliptic problem with logistic source, \cite{wang2014quasilinear} proved global existence of bounded weak solutions in the diffusion-dominated range
\(
m>2-\frac{2}{N},
\)
and also obtained boundedness in the range
\(
1<m\le 2-\frac{2}{N}
\)
under a sufficiently strong logistic damping condition. Thus, logistic damping can compensate for weaker diffusion and extend global boundedness into regimes where diffusion alone may not be sufficient. In contrast, in the absence of logistic damping, that is, when \(a=b=0\), \cite{cieslak2008finite} showed that finite-time blow-up may occur for suitable initial data in the subcritical diffusion range
\(
m<2-\frac{2}{N}.
\)
More precisely, there are initial data for which the corresponding solution satisfies
\begin{equation}\label{blowup}
    \limsup_{t\nearrow T_{\max}}\|u(\cdot,t)\|_{L^\infty(\Omega)}=\infty
\end{equation}
for some \(T_{\max}<\infty\). For the parabolic-parabolic problem with logistic source, \cite{yang2015boundedness} established the boundedness of nonnegative weak solutions under a sufficiently strong logistic damping condition, while \cite{zhang2015boundedness} obtained global bounded classical solutions under the condition
\(
m>\frac{2N}{N+2}.
\)
In the semilinear case \(m=1\), Tello and Winkler \cite{tello2007chemotaxis} established the existence of a global weak solution, possibly unbounded, for any \(b>0\). Their result further highlights the role of logistic damping in preventing finite-time blow-up.

\smallskip

On \(\mathbb R^N\), the system \eqref{main-eq} has also been studied for integrable and bounded initial data in the absence of logistic source terms, namely when \(a=b=0\). In the parabolic-elliptic case \((\tau=0)\), Sugiyama \cite{sugiyama2007time} established the existence of globally bounded weak solutions when
\(
m>2-\frac{2}{N},
\)
and also obtained global existence for sufficiently small initial data, in a suitable sense, in the range
\(
1<m\le 2-\frac{2}{N}.
\)
On the other hand, Sugiyama \cite{sugiyama2006global} constructed suitable initial data for which solutions blow up in finite time when
\(
1<m<2-\frac{2}{N}
\quad\text{and}\quad N\ge 3,
\)
in the sense of \eqref{blowup}. For the parabolic-parabolic case \((\tau=1)\), Sugiyama \cite{sugiyama2007time} proved the existence of global weak solutions for \(m\ge 2\). This condition was later improved by Ishida and Yokota \cite{ishida2012global}, who established global existence for
\(
m>2-\frac{2}{N}.
\)

For both \(\tau\in\{0,1\}\), using the Simon--Dubinski\u{\i} compactness theorem, we show that for every \(m>1\), there exists at least one global weak solution of \eqref{main-eq}, which may in general be unbounded. Furthermore, we prove the existence of globally bounded weak solutions whenever \(b>b_m\). In the case \(\tau=1\), we have \(b_m=0\) when \(m>\frac{2N}{N+2}\), while \(b_m>0\) when \(1<m\le \frac{2N}{N+2}\). Similarly, in the case \(\tau=0\), we have \(b_m=0\) when \(m>2-\frac{2}{N}\), and $b_m>0$ when \(1<m\le 2-\frac{2}{N}\). Thus, the logistic source provides a damping mechanism strong enough to guarantee global boundedness.

Our proof of the existence of globally bounded nonnegative weak solutions is based on weighted energy and localization techniques, together with a continuity-type argument introduced in \cite{hassan2025global} for the study of globally bounded weak solutions in the case where the chemical substance is consumed by the mobile species:
\begin{equation}
\label{special-eq3}
\begin{cases}
u_t = \Delta u^m - \chi \nabla \cdot (u \nabla v) + u(a - b u),\quad & x \in \mathbb{R}^N, \,\, t>0, \\
v_t = \Delta v - uv, \quad & x \in \mathbb{R}^N,\,\,  t>0,\\
u(0,x) = u_0(x),\,\, v(0, x) = v_0(x), \quad & x\in \R^N,
\end{cases}
\end{equation}
where the initial data satisfy \eqref{initial-cond-eq}.

The main difference between the linear production model and \eqref{special-eq3} is that, for \eqref{special-eq3}, the comparison principle immediately yields a uniform \(L^\infty\)-bound for \(v\). This boundedness is useful for controlling the local \(L^{2r}\)-norm of \(\nabla v\) by the local \(L^{r'}\)-norm of \(u\), for some \(r'>r>1\) (see Lemma 3.1 in \cite{hassan2025global}), which in turn plays an important role in establishing a local \(L^p\)-bound for \(u\), for some \(p>1\). In contrast, for \eqref{main-eq}, \(\|v\|_{L^\infty}\) is not always bounded, and therefore the terms involving \(v\) must be treated more carefully when proving boundedness of the local \(L^p\)-norm of \(u\). In this setting, the diffusion term is particularly helpful, especially when \(m\) is close to \(1\).

H\"older continuity is the highest regularity generally expected for globally bounded weak solutions of \eqref{main-eq}. It is therefore natural to establish uniqueness within this class. In the absence of logistic source terms, that is, when \(a=b=0\), uniqueness results are known under additional regularity and integrability assumptions on the initial data. In particular, Kawakami and Sugiyama \cite{kawakami2016uniqueness} established uniqueness of weak solutions of \eqref{main-eq} for \(\tau=0\), in the class of H\"older continuous functions, when
\(
m>\max\{\frac12-\frac1N,0\}.
\)
Miura and Sugiyama \cite{miura2014uniqueness} proved uniqueness of weak solutions to \eqref{main-eq}, for \(\tau\in\{0,1\}\), by a vanishing-viscosity duality method under the same condition on \(m\) and in the same H\"older class. Using the duality method applied in \cite{hassan2025global}, we establish uniqueness of weak solutions for the linear production case \eqref{main-eq} when \(1<m\le 3\), \(a,b>0\), and without any integrability assumptions on the initial data. By the same argument, uniqueness also holds for any \(m>1\) in the absence of logistic source terms.
\subsection{Definitions and Main Results}

In this subsection, we introduce the definitions of weak solutions of \eqref{main-eq}
 and state the main results of the paper.

\begin{defin}
\label{D.1}
Let \(m>1\), \(\tau\in\{0,1\}\) and \(T>0\), and let
\(u_0\in L^\infty(\mathbb R^N)\). When \(\tau=1\), let
\(v_0\in W^{1,\infty}(\mathbb R^N)\). When \(\tau=0\), no initial
condition is imposed on \(v\), and all terms involving \(\tau v_0\) are
understood to be zero.

A pair \((u,v)\) defined a.e. in
\((0,T)\times\mathbb R^N\) is called a weak solution of
\eqref{main-eq} on \([0,T)\) if the following conditions hold:
\begin{itemize}
    \item[(1)] 
    \(u\in L^2(0,T;L^2_{\rm loc}(\mathbb R^N))\) and
    \(\nabla u^m\in L^1(0,T;L^1_{\rm loc}(\mathbb R^N))\);

    \item[(2)]
    \(v\in L^1(0,T;W^{1,1}_{\rm loc}(\mathbb R^N))\) and
    \(u\nabla v\in L^1(0,T;L^1_{\rm loc}(\mathbb R^N))\);

    \item[(3)]
    For every \(\phi\in C_c^1([0,T)\times\mathbb R^N)\), we have
    \begin{align*}
    &\int_0^T\int_{\mathbb R^N} u\phi_t\,dxdt
    +\int_{\mathbb R^N}u_0(x)\phi(0,x)\,dx \\
    &\quad =
    \int_0^T\int_{\mathbb R^N}
    \left(
    \nabla u^m\cdot\nabla\phi
    -\chi u\nabla v\cdot\nabla\phi
    -au\phi
    +bu^2\phi
    \right)\,dxdt,
    \end{align*}
    and
    \begin{align*}
    &\tau\int_0^T\int_{\mathbb R^N} v\phi_t\,dxdt
    +\tau\int_{\mathbb R^N}v_0(x)\phi(0,x)\,dx  =
    \int_0^T\int_{\mathbb R^N}
    \left(
    \nabla v\cdot\nabla\phi
    +\lambda v\phi
    -\mu u\phi
    \right)\,dxdt.
    \end{align*}
\end{itemize}

A global weak solution of \eqref{main-eq} is a pair \((u,v)\) of
functions defined a.e. in \((0,\infty)\times\mathbb R^N\) such that
\((u,v)\) is a weak solution of \eqref{main-eq} on \([0,T)\) for every
\(T>0\).
\end{defin}

Now we state our main results. The first theorem established existence of a global non-negative weak solution which may not stay bounded, for any $m>1$, $b>0$. 

\begin{tm}[Global Existence]\label{main-thm1}
    Let $m>1$, \(a,b, \lambda, \mu >0\) and $\chi\in \R$. Let $T>0$, $\tau\in \{0, 1\}$ and let $u_0\in L^\infty(\R^N)$ and $\tau v_0\in W^{1,\infty}(\R^N)$ such that $u_0,\tau v_0\geq 0$. Then there exists at least one non-negative global weak solution of \eqref{main-eq} satisfying the following 
   \[
u\in L^\infty\big(0,T;L^{m+1}_{\mathrm{loc}}(\R^N)\big)
\cap L^{m+2}_{\mathrm{loc}}\big((0,T)\times \R^N\big),
\quad
\nabla u^m\in L^2_{\mathrm{loc}}\big((0,T)\times \R^N\big),
\]
and
\[
v,\ \nabla v \in L^\infty\big(0,T;L^{m+1}_{\mathrm{loc}}(\R^N)\big), \quad u\nabla v \in L^p_{\mathrm{loc}}\big((0,T)\times \R^N\big)
\]
for some \(p\in(1,2)\).
\end{tm}
The next theorem provide a condition on $m$ and the logistic source $b$, under which we can get a weak solution of \eqref{main-eq} that is globally bounded. We denote $\Omega_t:=[0,t]\times\R^N$ for any $t>0$.

\begin{tm}[Global Existence/Boundedness]
\label{main-thm}
Assume $m>1$, \(a,b, \lambda, \mu >0\) and $\chi\in \R$. Let $T>0$, $\tau\in \{0, 1\}$ and let $u_0\in L^\infty(\R^N)$ and $\tau v_0\in W^{1,\infty}(\R^N)$ such that $u_0,\tau v_0\geq 0$. Assume that
\begin{equation*}
    b> b_m := \begin{cases}
        \left(\inf_{{p} >\max\{1, \frac{N(2-m)}{2m}\}}\frac{p-1}{p}\left(C_{p+1,N}\right)^{\frac{1}{p+1}}\right)|\chi|\mu, \quad & 1<m\le \frac{2N}{N+2},\\[1mm]
        0, & m> \frac{2N}{N+2},
    \end{cases}
\end{equation*}
when \(\tau=1\), and
\begin{equation*}
    b> b_m := \begin{cases}
        \frac{\mu|\chi|((2-m)N -2)_+}{(2-m)N}, \quad & 1<m\le 2-\frac2N,\\[1mm]
        0, & m> 2-\frac2N,
    \end{cases}
\end{equation*}
when \(\tau = 0\). Then \eqref{main-eq} admits a globally bounded nonnegative weak solution $(u,v)$ on
$\Omega_T$. Moreover, for every $p\ge m$,
\[
u(t,\cdot)\in L^p_{\rm loc}(\R^N),
\qquad 
\nabla u^{\frac{p+m}{2}}\in L^2_{\rm loc}(\Omega_T),
\qquad t\in[0,T],
\]
with a bound depending only on
\(
m,|\chi|,a,b,N,p,\lambda,\mu,\|u_0\|_\infty,
\|\tau v_0\|_{W^{1,\infty}},
\)
and the diameter of the local spatial domain, but not on $T$. In addition, there exists a constant
$C>0$, independent of $T$, such that
\[
\|u\|_{L^\infty(\Omega_T)}
+
\|v\|_{W^{1,\infty}(\Omega_T)}
\le C.
\]
Furthermore,
\(
u\in C\big((0,T);L^2_{\rm loc}(\R^N)\big).
\)
\end{tm}

The last theorem is on H\"older continuity and uniqueness of bounded non-negative weak solutions of \eqref{main-eq}.

\begin{tm}[H\"older Regularity and Uniqueness]\label{uniqueness}
Let \(m>1\), \(a,b,\lambda,\mu>0\), \(\chi\in\mathbb R\), and
\(\tau\in\{0,1\}\). Let \((u,v)\) be a globally bounded nonnegative weak
solution of \eqref{main-eq} with initial data satisfying \eqref{initial-cond-eq}.
Assume that, for every \(T>0\),
\(
u^m\in C_{\mathrm{loc}}\big((0,T);L^2_{\mathrm{loc}}(\mathbb R^N)\big)
\cap
L^2_{\mathrm{loc}}\big((0,T);W^{1,2}_{\mathrm{loc}}(\mathbb R^N)\big),
\)
and
\(
u\nabla v\in L^2_{\mathrm{loc}}\big((0,T)\times\mathbb R^N\big).
\)
Then the following hold:
\begin{enumerate}
    \item \textup{(H\"older continuity)}
    For each \(\eta>0\), there exists \(\alpha\in(0,1)\) such that
    \(u\) is uniformly H\"older continuous in
    \([\eta,\infty)\times\mathbb R^N\). Moreover, if $u_0$ is uniformly H\"{o}lder continuous, then $u$ is uniformly H\"{o}lder continuous in $[0,\infty)\times\bbR^N$ and $v\in C^{2+\alpha}([0,\infty)\times\R^N)$ when $\tau =0$. If in addition $v_0\in C^{2+\alpha}$ when $\tau =1$, then $v\in C^{1+\alpha/2, 2+\alpha}([0,\infty)\times\R^N )$.

    \item \textup{(Uniqueness)}
    Suppose that \((u',v')\) is another globally bounded nonnegative
    weak solution of \eqref{main-eq} with the same initial data, satisfying
    the same assumptions as \((u,v)\). If \(u_0\in C^\alpha(\mathbb R^N)\)
    and \(\tau v_0\in C^{2+\alpha}(\mathbb R^N)\), then for $1<m\le 3$,
    \[
    u'=u
    \qquad\text{and}\qquad
    v'=v
    \quad\text{a.e. in }(0,\infty)\times\mathbb R^N.
    \]
\end{enumerate}
\end{tm}

We conclude this section with the following remarks on our results.

\begin{rk}
\begin{itemize}
    \item[(1)] \noindent\textbf{Novelty.}
To the best of our knowledge, this is the first work to establish the existence of a weak solution for every \(m>1\), as well as the existence of a globally bounded weak solution under a suitable condition on \(b\), for initial data that are not necessarily integrable. Theorems~1.1 -- Theorem~1.3 also hold in the parabolic--parabolic case for any \(\tau>0\). We take \(\tau=1\) only for simplicity.
    
    \item[(2)] \textbf{Comparison with the semilinear case.} When \(m=1\), the condition \(b>b_m\) in Theorem 1.1 coincides with the corresponding condition in the semilinear case when \(m=1\); see \cite{hassan2024chemotaxis,issa2020pointwise} for the case \(\tau=1\), and \cite{hassan2024chemotaxis, salako2017global,tello2007chemotaxis} for the case \(\tau=0\). Moreover, when \(N=1,2\), we have \(b_m=0\), which means that a globally bounded nonnegative weak solution exists in one and two dimensions for any $b>0$. For \(N\ge 3\), the condition \(b>b_m\) is sharper than that in the semilinear case, indicating that the stronger diffusion is beneficial in this setting.

    \item [(3)]\noindent\textbf{Remark on uniqueness.}
We prove uniqueness for any nonnegative globally bounded weak solution of
\eqref{main-eq} that is H\"older continuous up to the initial time, provided
\(1<m\le 3\). Note that Theorem~1.1 establishes only the existence of a weak
solution, which may be unbounded and may fail to be H\"older continuous, when
\(1<m\le m_\tau\), where
\[
m_\tau=\frac{2N}{N+2}\quad\text{if }\tau=1,
\qquad
m_\tau=2-\frac{2}{N}\quad\text{if }\tau=0.
\]
   
\end{itemize}
\end{rk}

The rest of the paper is organized as follows. Section \ref{ss.2.1} collects several preliminary lemmas. Section 3 is devoted to uniform local $L^p$ estimates. In Section 4, we prove the existence of global weak solutions and establish Theorems \ref{main-thm1} and \ref{main-thm}. Finally, in Section 5, we prove uniqueness of weak solutions and establish Theorem \ref{uniqueness}.

\section{Preliminary}\lb{ss.2.1}
In this section, we present some preliminary lemmas that will be used in the rest of the paper. These include $L^p-L^q$ estimates for the analytic semigroup generated by $\Delta-\lambda I$ on $L^p$,  a class of useful decay functions,  a  lemma on maximal regularity for parabolic equations on $\R^N$, and some result on a perturbed problem of \eqref{main-eq} that would be useful in the proof of Theorem \ref{main-thm1} and Theorem \ref{main-thm}.

\subsection{Analytic Semigroup $e^{(\Delta-\lambda I)t}$ and Maximal Regularity for Parabolic Equations}

We start with some basic properties of the analytic semigroup $T_p(t)$  generated by $\Delta-\lambda I$ on $L^p(\mathbb{R}^N)$ $(p\ge 1$). This is defined by
\begin{equation}
\label{semigroup-eq}
(T_p(t)u)(x)=e^{-\lambda t}(G(t, \cdot)\ast u)(x)= \int_{\R^{N}}e^{-\lambda t}G(t, x-y)u(y)dy
\end{equation}
for every $u\in L^p(\mathbb{R}^N)$, $t> 0$,  and $x\in\R^N$, where  $G(t,x)$ is the heat kernel defined by
\begin{equation}
\label{heat-kernel}
G(t,x)={(4\pi t)^{-\frac{N}{2}}}e^{-\frac{|x|^{2}}{4t}}.
\end{equation}

\begin{lem}
Let \(1\leq p\leq \infty\), and let \(T_p(t)\) denote the semigroup generated by
\(\Delta-\lambda I\) on \(L^p(\mathbb R^N)\). Then the following properties hold.

\begin{enumerate}
    \item For \(1\leq p<q\leq \infty\), there exists a constant \(C_{p,q}>0\) such that, for all \(t>0\) and all \(u\in L^p(\mathbb R^N)\),
    \begin{equation}\label{Lp Estimates-2}
    \|T_p(t)u\|_{L^{q}(\mathbb{R}^{N})}
    \leq C_{p,q} e^{-\lambda t}
    t^{-\left(\frac{1}{p}-\frac{1}{q}\right)\frac{N}{2}}
    \|u\|_{L^{p}(\mathbb{R}^{N})}.
    \end{equation}
    Moreover,
    \begin{equation}\label{Lp Estimates-3}
    \|\nabla T_p(t)u\|_{L^{q}(\mathbb{R}^{N})}
    \leq C_{p,q} e^{-\lambda t}
    t^{-\frac{1}{2}-\left(\frac{1}{p}-\frac{1}{q}\right)\frac{N}{2}}
    \|u\|_{L^{p}(\mathbb{R}^{N})}.
    \end{equation}

    \item If \(u\in L^p(\mathbb R^N)\cap L^\infty(\mathbb R^N)\), then
    \(T_p(t)u\in L^\infty(\mathbb R^N)\) and, for all \(t\geq 0\),
    \begin{equation}\label{L-infty- Estimates-1}
    \|T_p(t)u\|_{L^{\infty}(\mathbb{R}^{N})}
    \leq e^{-\lambda t}\|u\|_{L^{\infty}(\mathbb{R}^{N})}.
    \end{equation}
\end{enumerate}
\end{lem}

The first estimate follows from the standard \(L^p\)-\(L^q\) estimates for the heat semigroup. The gradient estimate follows similarly by differentiating the heat kernel. The second estimate follows from the positivity and unit mass of the heat kernel.

Next is a maximal regularity lemma that would be useful in controlling the terms involving $v$
for the parabolic-parabolic case $(\tau = 1)$.

\begin{lem}
    \label{maximal-regularity-lm}
        Let    $v_0\in W^{1,\gamma}(\R^N)\cap L^\infty(\R^N)$.  There exists $C_{\gamma,N}$  such that for any 
  $T\in (0,\infty)$, if  $g \in L^\gamma((0,T), L^\gamma(\R^N))$ and $v(\cdot,\cdot)\in W^{1,\gamma}((0,T),L^{\gamma}(\R^N))\cap L^{\gamma}((0,T), W^{2,\gamma}(\R^N))$ solves  the following initial boundary value problem,
    \begin{equation}
    \label{pdelaplace}
    \begin{cases}
     v_t =\Delta v -  \lambda v + g,\quad &x\in \R^N,\,\,  0<t<T,\cr
    v(0,x) = v_0(x),\quad &x\in \R^N,
    \end{cases}
    \end{equation}
    then 
    \beq\lb{maximal-regularity-eq1-1}
    \begin{aligned}
   &\quad\, \int_{0}^T \int_{\R^N}e^{ \lambda\gamma t}\left( |v (t,x)|^\gamma +|\nabla v(t,x)|^\gamma +|\Delta v(t,x)|^\gamma\right)dxdt\\
& \le C_{\gamma,N}\left[ \int_{0}^T \int_{\R^N}e^{\lambda\gamma  t}|g(t,x)|^{\gamma}dx dt+ T\left( \|v_0(\cdot)\|^{\gamma}_{L^\gamma(\R^N)}+\|\nabla v_0(\cdot)\|_{L^\gamma(\R^N)}^\gamma\right)\right] .
    \end{aligned}
    \eeq
    \end{lem}
 The proof of the above lemma for  $\lambda =1$ can be found in \cite[Lemma 2.3]{hassan2025global}. The general case follows by same argument.

\subsection{A Class of Exponential Decay Functions}

The first lemma in this subsection is a useful exponential decay function that would be helpful in controlling non-integrability. The proof of this lemma can be found in \cite[Lemma 2.1]{hassan2025global}.

\begin{lem}
\label{psi-lm}
Take a smooth decreasing function $f$ on $\R$ such that
\[
f(r)=1\quad\text{when }r\leq N\quad\text{ and }\quad f(r)=2^{-1} e^{N+1-r}\quad\text{when }r\geq N+1.
\]
For each fixed $\kappa\in (0,1)$ and for some $\gamma\in (0,1)$, define
\beq\lb{varphi}
\psi(x):=\psi_\kappa(x)=f(\gamma\kappa|x|).
\eeq
There exists dimensional constants $C,\gamma>0$ such that for any $\kappa\in (0,1)$, $\psi$ from 
 \eqref{varphi} satisfies for all $x\in\R^N$,
\begin{equation}\label{phi1}
 0<\psi(x)\leq 1,\quad |\nabla \psi(x)|\le \k\, \psi(x),\quad |D^2 \psi(x)|\le \k^2\,\psi(x),
\end{equation}
\beq\lb{phi2}
\psi(x)\leq C\psi(y)\quad \text{ whenever }\quad |x-y|\leq \kappa^{-1},
\eeq
and
\beq\lb{phi3}
\kappa^{N}\int_{\R^N}\psi(x) dx,\quad \sum_{\kappa z\in \Z^N}\psi(z)\leq C.
\eeq

\end{lem}

Next, we state a localized Sobolev interpolation inequality in terms of the exponential decay function from the previous lemma that will be useful for deriving local \(L^p\)-estimates. Its proof appears in \cite[Lemma 2.2]{hassan2025global}.

\begin{lem}
\label{I1-lm}
Assume that \(N\ge 3\), and let
\(
2^*:=\frac{2N}{N-2}>2
\)
denote the Sobolev conjugate exponent of \(2\). Then there exists a constant \(C>0\), depending only on the dimension, such that for any \(\delta,\kappa>0\) and \(r>1\), we have
\begin{align*}
\kappa^2\int_{\R^N}u^{2r}\psi^2
&\le \delta\|\nabla(u^{r}\psi)\|_{L^2(\R^N)}^2
+C\kappa^{2+2\theta_r}\delta^{-\theta_r}\left[\int_{\R^N}u^{r+1}\psi^{q_r}\,dx\right]^{2/q_r},
\end{align*}
where \(u=u(x)\) is any function such that \(u^r\) is locally uniformly bounded in \(W^{1,2}\), \(\psi\) is as in Lemma \ref{psi-lm}, and
\[
q_r:=\frac{r+1}{r}\in(1,2),
\qquad
\theta_r:=\frac{N(r-1)}{2(r+1)}>0.
\]
\end{lem}

\subsection{Perturbed Chemotaxis Models}
We prove the existence of global weak solutions to \eqref{main-eq} by approximating them with classical solutions of the following perturbed problem:
\begin{equation}
\label{main-perturbed-eq}
\begin{cases}
u_t = m\nabla\cdot \big((\eps+u)^{m-1}\nabla u\big) - \chi \nabla \cdot (u \nabla v) + f(u),\quad & x \in \mathbb{R}^N, \,\, t>0, \\
\tau v_t = \Delta v - \lambda v + \mu u, \quad & x \in \mathbb{R}^N,\,\,  t>0,\\
u(0,x) = u_0(x),\,\, \tau v(0, x) = \tau v_0(x), \quad & x\in \R^N,
\end{cases}
\end{equation}
where \(\eps\in(0,1)\) and $f(u) \le au - bu^2$, and initials are smooth enough. We refer to \eqref{main-perturbed-eq} as the perturbed problem associated with \eqref{main-eq}. Throughout the rest of the paper, we use the notation \(\tau v_0\) to encode the initial condition for \(v\) in both cases. More precisely, when \(\tau=1\), the condition \(\tau v(0,x)=\tau v_0(x)\) means \(v(0,x)=v_0(x)\). When \(\tau=0\), the second equation is elliptic and no initial condition is imposed on \(v\); in this case, \(\tau v(0,x)=\tau v_0(x)\) is understood as the identity \(0=0\).

In the following, we write \(v\in C^{1+\alpha/2,\,2+\alpha}\) to mean that \(v\) is \(C^{1+\alpha/2}\) in time and \(C^{2+\alpha}\) in space. Likewise,
\(
u \in C^\alpha([0,T)\times\R^N)\cap C^{1+\alpha/2,\,2+\alpha}((0,T)\times\R^N)
\)
means that \(u\) is H\"older continuous in space and time up to the initial time, and is \(C^{1+\alpha/2}\) in time and \(C^{2+\alpha}\) in space for positive times.

\begin{defin}
\label{D.2}
Let \(\tau\in\{0,1\}\). Assume that $u_0\in C^{1+\alpha}_{\mathrm{unif}}(\R^N)$, $\tau v_0\in C^{2+\alpha}_{\mathrm{unif}}(\R^N)$.
A pair \((u_\eps,v_\eps)\) of nonnegative functions defined on \([0,T)\times\R^N\) is called a classical solution of \eqref{main-perturbed-eq} with \(\eps>0\) on \([0,T)\) if the following hold:
    \[
    u_\eps \in C^\alpha([0,T)\times\R^N)\cap C^{1+\alpha/2,\,2+\alpha}((0,T)\times\R^N),
    \]
    and, moreover,
    \[
    v_\eps \in C\big([0,T);C^{2+\alpha}(\R^N)\big)
    \quad \text{if } \tau=0,
    \]
    while
    \[
    v_\eps \in C^{1+\alpha/2,\,2+\alpha}((0,T)\times\R^N)
    \cap C\big([0,T);C^{2+\alpha}(\R^N)\big)
    \quad \text{if } \tau=1.
    \]
\(u_\eps(0,\cdot)=u_0\), \(\tau v_\eps(0,\cdot)=\tau v_0\), and \eqref{main-perturbed-eq} is satisfied in the classical sense in \((0,T)\times\R^N\). By a global classical solution we mean a classical solution on $[0, \infty).$
\end{defin}

Next we would prove a lemma which is on a priori estimate of classical solutions of the  perturbed problem \eqref{main-perturbed-eq}. The lemma shows that if the local \(L^p\)-norm of \(u_\varepsilon\) is bounded for some sufficiently large \(p\), then one can obtain \(W^{1,\infty}\)-bounds for \(v_\varepsilon\), \(L^\infty\)-bounds for \(u_\varepsilon\).
\begin{lem}
\label{v-bound-lm} Let $\eps\in(0,1)$, $a,b, \mu, \lambda>0$, $\chi\in\bbR$, $m>1$ and $T>0$, and let $u_0\in L^\infty(\R^N)$ and $\tau v_0\in W^{1,\infty}(\R^N)$ such that $u_0,\tau v_0\geq 0$. Suppose that $(u_\eps, v_\eps)$ is a non-negative classical solution to \eqref{main-perturbed-eq} on $[0,T)$, with initial data $(u_0, v_0)$.
For any $p>N$ and $C_1>0$, independent of $T$ and $\eps$ (but could depend on $p$)  if
\begin{equation}\label{lplocal}
\sup_{t\in [0,T],x_0\in\R^N}\int_{B(x_0,1)} u_\eps^p(t,x)dx\leq C_1,
\end{equation}
then the following holds;
\begin{itemize}
    \item [1)] (Uniform $W^{1,\infty}$ estimate for $v_\eps$) There is a constant  $C$ depending only on $p$, $N$,  $C_1$, and $ \|\tau v_0\|_{W^{1,\infty}}$ such that 
\begin{equation*}
\sup_{t\in [0,T],x\in\R^N}|\nabla v_\eps(t,x)| + \sup_{t\in [0,T],x\in\R^N}|v_\eps(t,x)|<C.
\end{equation*}

\item[2)] (Uniform $L^\infty$ estimate for $u_\eps$)  $u_\eps$ is uniformly bounded in $[0,T]\times\R^N$ with the bound depending on general constants ($m,|\chi|,a, b,N,p$, $\lambda$, $\mu$, $\|u_0\|_\infty$,  $\|\tau v_0\|_{W^{1,\infty}}$), but  independent of $\eps$ and $T$. 

\item[(3)] For any $p\ge m$ and $t\in [0,T]$, we have that $\nabla  (\eps+u_\eps)^{\frac{p+m}{2}}\in  L^2_{\rm loc}(\Omega_T)$,
with a bound depending only on general constant and the diameter of the local spatial domain (but independent of $T$ and $\eps$).
\end{itemize}

\end{lem}

\begin{proof}[Proof of Lemma \ref{v-bound-lm}]
Let \(\psi\) be as in Lemma \ref{psi-lm}, with \(\kappa>0\) to be chosen sufficiently small. Throughout the proof, \(C\) and $c$ denotes a positive constant independent of \(\varepsilon\), \(t\), and \(x_0\), while \(C_\kappa\) may depend on \(\kappa\).

\noindent{\bf Proof of property (1) when $\tau = 1$:} 

First, note that \(v_\eps\psi\) satisfies the equation
\begin{align}
\label{v-phi-eq}
(v_\eps\psi)_t=\Delta(v_\eps\psi)-\lambda v_\eps\psi +\left(\mu u_\eps \psi-2\nabla v_\eps\cdot \nabla\psi - v_\eps\Delta \psi \right).
\end{align}
and \(v_0\psi\in W^{1,p}(\R^N)\cap W^{1,\infty}(\R^N)\) for any \( p>1\). Hence, for any \(p\le p'\le \infty\),
\begin{align} \label{vlp'}
\|v_\eps(t,\cdot)\psi\|_{L^{p'}}\nonumber
&\le e^{- \lambda t}\|v_0\psi\|_{ L^{p'}}+ C\sup_{r\in [0,T]}\|u_\eps(r,\cdot)\psi\|_{L^{p}} \int_0^t  e^{-\lambda (t-s)}
(t-s)^{-\big(\frac{1}{p}-\frac{1}{p'}\big)\frac{N}{2}} ds\nonumber\\
&\quad + \kappa \sup_{r\in [0,T]}\Big( 2 \| \nabla v_\eps(r,\cdot) \psi\|_{L^{p}}  +   \|v_\eps(r,\cdot)\psi \|_{L^{p}}\Big)\int_0^t  e^{-\lambda (t-s)} (t-s)^{-\big(\frac{1}{p}-\frac{1}{p'}\big)\frac{N}{2}}ds\nonumber\\
& \le C_\k  + \kappa \left(\sup_{r\in [0,T]}\| \nabla v_\eps(r,\cdot) \psi\|_{L^{p}}  +   \sup_{r\in [0,T]}\|v_\eps(r,\cdot)\psi \|_{L^{p}}\right) \int_0^t  e^{-\lambda (t-s)} (t-s)^{-\big(\frac{1}{p}-\frac{1}{p'}\big)\frac{N}{2}}ds.
\end{align}
Where we used \eqref{lplocal} to get the last inequality. Also, using  \eqref{Lp Estimates-2}-\eqref{L-infty- Estimates-1}  and \(v_0\psi\in W^{1,p}(\R^N)\cap W^{1,\infty}(\R^N)\)  we have that for any \(p\le p'\le\infty\), 
\begin{align}
\label{A-eq}
\| (\nabla  v_\eps (t,\cdot))\psi\|_{L^{p'}}
&\le   \int_0^t e^{-\lambda(t-s)}(t-s)^{-\frac{1}{2}-\big(\frac{1}{p}-\frac{1}{p'}\big)\frac{N}{2}}\| u_\eps\psi -2\nabla v_\eps\cdot\nabla \psi-v_\eps\Delta\psi  \|_{L^p}\,ds\nonumber\\
&\quad +
 \kappa \|v_\eps(t,\cdot) \psi\|_{L^{p'}}
 + C_\k\nonumber\\
 &\lesssim \k \left(\sup_{r\in [0,T]}\| \nabla v_\eps(r,\cdot) \psi\|_{L^{p}}  +   \sup_{r\in [0,T]}\|v_\eps(r,\cdot)\psi \|_{L^{p}}\right) \int_0^t  e^{-\lambda (t-s)} (t-s)^{-\big(\frac{1}{p}-\frac{1}{p'}\big)\frac{N}{2}}ds \nonumber\\
 &\quad +
 \kappa \|v_\eps(t,\cdot) \psi\|_{L^{p'}}
 + C_\k.
\end{align}
 Now for $p'= p$, taking supremum in \(t\in [0,T]\) in \eqref{vlp'} and \eqref{A-eq} then adding the two equations, gives that for $\k$ sufficiently small,
\begin{equation}
\label{A-eq1}
\sup_{0\le t\le T}\|(\nabla v_\eps(t,\cdot))\psi\|_{L^p} + \sup_{0\le t\le T}\|v_\eps(t,\cdot)\psi\|_{L^p}\le C_\kappa.
\end{equation}
In the rest of the proof, we fix one such \(\kappa\) and may drop it from the notation of \(C_\kappa\). 

Taking $p' = \infty$, yields that for $p>N$,
\begin{align*}
\|(\nabla v_\eps(t,\cdot))\psi\|_{L^\infty} + \|v_\eps(t,\cdot)\psi\|_{L^\infty}
&\leq CC_1+C\kappa \left(\sup_{r\in [0,T]}\| \nabla v_\eps(r,\cdot) \psi\|_{L^{p}}  +   \sup_{r\in [0,T]}\|v_\eps(r,\cdot)\psi \|_{L^{p}}\right)\\
&\le C\quad \forall \, t\in [0,T].
\end{align*}
 The last inequality follows from \eqref{A-eq1}. 
Replacing \(\psi(\cdot)\) by \(\psi(\cdot -x_0)\) yields
$$
\| (\nabla  v_\eps (t,\cdot))\psi(\cdot-x_0)\|_{L^{\infty}} + \|   v_\eps (t,\cdot)\psi(\cdot-x_0)\|_{L^{\infty}}\le C
$$
uniformly for all \(t\geq 0\) and \(x\in\bbR^N\).
This implies that
$$
\sup_{0\le t\le T} \|\nabla v_\eps(t,\cdot)\|_{L^\infty} + \sup_{0\le t\le T} \|v_\eps(t,\cdot)\|_{L^\infty}\le C
$$
for \(C\) depending only on \(p\), \(N\), \(C_1\), \(\mu\), \(\lambda\), and \(\|v_0\|_{W^{1,\infty}}\).

\medskip

\noindent{\bf Proof of property (1) when \(\tau = 0\):}
In this case, \(v_\eps\psi\) satisfies
\begin{equation*}
\Delta (v_\eps\psi)-\lambda v_\eps \psi -2\nabla v_\eps\cdot\nabla \psi-v_\eps\Delta \psi +\mu u_\eps\psi=0.
\end{equation*}
Hence, using the resolvent \((\Delta -\lambda I)^{-1}\) on \(\R^N\), we obtain
\begin{equation}
\label{main-eq3-psi}
v_\eps(t,\cdot)\psi(\cdot)=\int_0^\infty e^{(\Delta -\lambda I)s}\Big(-2\nabla v_\eps(t,\cdot)\cdot\nabla \psi-v_\eps(t,\cdot)\Delta \psi+\mu u_\eps(t,\cdot)\psi\Big)\,ds.
\end{equation}
We have that for any \(p\le p'\le\infty\) and \(t\in[0,T]\),
\begin{align}
\label{est-tau0'}
\|v_\eps(t,\cdot)\psi(\cdot)\|_{L^{p'}}
&\le C \kappa \Big(\|\nabla v_\eps (t,\cdot) \psi\|_{L^p}+
\|v_\eps(t,\cdot) \psi\|_{L^p}\Big) \int_0^\infty e^{-\lambda s} s^{-\big(\frac{1}{p}-\frac{1}{p'}\big)\frac{N}{2}}\,ds\nonumber\\
& \quad+ C  \|u_\eps(t,\cdot)\psi\|_{L^p}  \int_0^\infty e^{-\lambda s} s^{-\big(\frac{1}{p}-\frac{1}{p'}\big)\frac{N}{2}}\,ds
\end{align}
and
\begin{align}
\label{est-tau0}
\|(\nabla v_\eps(t,\cdot)) \psi\|_{L^{p}}
&\le \kappa{\|v_\eps(t,\cdot)\psi\|_{L^{p}}}+C \|u_\eps(t,\cdot)\psi\|_{L^p}\int_0^\infty e^{-\lambda s}s^{-\frac{1}{2}} \,ds\nonumber\\
&\quad + C\kappa \Big(\|\nabla v_\eps(t,\cdot)\psi\|_{L^p}+\|v_\eps(t,\cdot)\psi\|_{L^p}\Big) \int_0^\infty e^{-\lambda s}s^{-\frac{1}{2}} \,ds.
\end{align}
From \eqref{est-tau0} and \eqref{est-tau0'}, we have
\begin{equation}\label{eq:3.11}
    \|(\nabla v_\eps(t,\cdot)) \psi\|^p_{L^{p}} + \| v_\eps(t,\cdot) \psi\|^p_{L^{p}} \le C \|u_\eps(t,\cdot)\psi\|^p_{L^p}+ 2C\kappa \Big(\|\nabla v_\eps(t,\cdot)\psi\|^p_{L^p}+\|v_\eps(t,\cdot)\psi\|^p_{L^p}\Big).
\end{equation}
Thus choosing \(\kappa\) small enough, we have
\begin{equation}\label{nablav-est}
    \int_{\R^N}|\nabla v_\eps|^p\psi + \int_{\R^N}v_\eps^p\psi\le C\int_{\R^N}u_\eps^p\psi.
\end{equation}
Then using the assumption \eqref{lplocal} together with a similar argument as in the case \(\tau =1\), we conclude that
$$
\sup_{0\le t\le T} \|\nabla v_\eps(t,\cdot)\|_{L^\infty} + \sup_{0\le t\le T} \|v_\eps(t,\cdot)\|_{L^\infty}\le C,
$$
for \(C\) depending only on \(p\), \(N\), \(C_1\), \(\mu\), \(\lambda\), and \(\|v_0\|_{W^{1,\infty}}\).

\medskip

\noindent{\bf Proof of property (2):} From property (1), it follows that \(\nabla v_\eps\) is a given \(L^\infty\)-vector field. This, together with assumption \eqref{lplocal} and \cite[Proposition 4.1]{hassan2025global}, implies that there exists a constant \(C>0\), depending only on the general constants but independent of \(T\) and \(\varepsilon\), such that
\[
\|u_\varepsilon\|_{L^\infty}\le C.
\]
The proof of \cite[Proposition 4.1]{hassan2025global} is based on a Moser iteration argument.

\medskip

\noindent{\bf Proof of property (3):} Multiplying the $u_\eps$ equation by $(u_\eps+\varepsilon)^p \psi$, then integrating give
\begin{align*}
&\quad\, \frac{1}{p+1}\frac{d}{dt}\int_{\R^N} (u_\eps +\eps)^{p+1}\psi  dx\nonumber \\
&\le \int_{\R^N} (u_\eps +\eps)^p\psi\nabla\cdot\left[m(u_\eps+\eps)^{m-1}\nabla u_\eps - \chi u_\eps\nabla v_\eps\right] 
   + u_\eps^{p}\psi(a u_\eps - b u_\eps^2)\nonumber \\
&\leq  -c\int_{\R^N} (u_\eps +\eps)^{p+m-2}|\nabla u_\eps|^2\psi
 + C\int_{\R^N}(u_\eps+\eps)^{p+m-1}|\nabla u_\eps||\nabla\psi| \nonumber\\
&\quad+ C \int_{\R^N} (u_\eps +\eps)^{p} |\nabla u_\eps||\nabla v_\eps|\,\psi
 + C\int_{\R^N} (u_\eps +\eps)^{p+1} |\nabla v_\eps||\nabla \psi|
 +\int_{\R^N}(u_\eps +\eps)^{p}(a u_\eps - b u_\eps^2)\psi\nonumber\\
&\leq  -c\int_{\R^N} (u_\eps +\eps)^{p+m-2}|\nabla u_\eps|^2\psi
 + C\kappa\int_{\R^N}(u_\eps+\eps)^{p+m}\psi
 + C \int_{\R^N} (u_\eps +\eps)^{p+2-m} |\nabla v_\eps|^2\,\psi\nonumber\\
&\quad + C\kappa\int_{\R^N} (u_\eps +\eps)^{p+1} |\nabla v_\eps|\psi
 +a\int_{\R^N}u_\eps ^{p+1}\psi.
\end{align*}
The last inequality holds by Young's inequality. Therefore, integrating this inequality in time gives
\begin{align*}
&\iint_{\Omega_t}  (u_\eps+\varepsilon)^{p+m-2} |\nabla u_\eps|^2 \psi \, dx \, ds\lesssim
\int_{\mathbb{R}^N} (u_{0}^{p+1}(x) + \varepsilon) \psi \, dx
+ \int_{\R^N}(u_\eps+\eps)^{p+m}\psi \nonumber\\
&\quad +\int_{\R^N} (u_\eps +\eps)^{p+2-m} |\nabla v_\eps|^2\,\psi
+ \int_{\R^N} (u_\eps +\eps)^{p+1} |\nabla v_\eps|\psi
+ \int_{\R^N}u_\eps ^{p+1}\psi\nonumber\\
&\quad\lesssim \kappa^{-N}.
\end{align*}

The last inequality follows from properties (1) and (2), the fact that \(p\ge m\), the regularity of the initial data, and Lemma \ref{psi-lm}. Indeed, properties (1) and (2) imply that
\[
v_\varepsilon \in W^{1,\infty}(\Omega_T)
\qquad\text{and}\qquad
\|u_\varepsilon\|_{L^\infty(\Omega_T)}\le C.
\]
Consequently,
\(
\nabla (u_\varepsilon+\varepsilon)^{\frac{p+m}{2}}
\)
is locally uniformly bounded in \(L^2(\Omega_T)\) for any $p\ge m$. This completes the proof of the lemma.
\end{proof}
The next Lemma is on uniform H\"older continuity of globally bounded classical solutions of the perturbed problem \eqref{main-perturbed-eq}.

\begin{lem}\label{holdercty-lem}
Let $\varepsilon\in(0,1)$, and let $(u_\varepsilon,v_\varepsilon)$ be a globally defined bounded non-negative classical solution of \eqref{main-perturbed-eq}. Assume that there exists a constant $C>0$, independent of $\varepsilon$, such that
\[
\|u_\varepsilon\|_{L^\infty((0,\infty)\times\mathbb R^N)}
+
\|v_\varepsilon\|_{L^\infty((0,\infty);W^{1,\infty}(\mathbb R^N))}
\le C .
\]
Then there exists $\alpha\in(0,1)$, depending only on $m$ and $C$, such that for every $t_0>0$ there exists a constant $C'>0$, depending only on
\(
m,\ |\chi|,\ a,\ b, \ \lambda, \ \mu, \ N,\ t_0,\ C,
\)
but independent of $\varepsilon$, such that
\[
\|u_\varepsilon\|_{C^\alpha([t_0,\infty)\times\mathbb R^N)}\le C' .
\]
If, in addition, $u_\varepsilon(0,\cdot)$ is H\"older continuous on $\mathbb R^N$, then $u_\varepsilon$ is H\"older continuous on $[0,\infty)\times\mathbb R^N$, with a bound depending also on the H\"older norm of $u_\varepsilon(0,\cdot)$.
\end{lem}
\begin{proof}
   Let \( z(t,x) := (u_\eps(t,x) + \varepsilon)^m \) and $ V(t,x) := \nabla v_\eps(t,x)$. Then  $z\ge \eps^m$ satisfies the equation
\begin{equation*}
(z^{\frac{1}{m}})_t = \Delta z-\chi\nabla\cdot((z^{\frac{1}{m}}-\eps) V) + a(z^{1/m} -\eps) - b(z^{1/m} -\eps)^2,\quad x\in\R^N,\,\, {t>0}.
\end{equation*} 
By assumption,
\(
\|V\|_{L^\infty(\Omega_T)} +
\|u_\varepsilon\|_{L^\infty(\Omega_T)}\le C.
\)
Therefore, the uniform interior H\"older continuity of \(z_\varepsilon\)  and hence of $u_\eps$, follows from \cite[Theorems 1.3]{black2026refining}, which are based on the De Giorgi--Nash--Moser iteration technique. Moreover, if \(u_0\in C^\alpha(\mathbb R^N)\), then by \cite[Theorems 1.9]{black2026refining}
\(
u_\varepsilon\in C^\alpha([0,T]\times\R^N)
\)
uniformly in \(\varepsilon\).
\end{proof}

The next lemma shows that we can control the local $L^{p+2}$ norm or $u_\eps$ by local $L^2$ norm of  $\nabla u_\eps^{\frac{p+m}{2}}$ for $1<m\le 2$. This would be particularly useful in establishing local $L^p$ estimate of the solution in the next section.

\begin{lem}\label{uLp+2}
Suppose that the assumptions of Lemma \ref{v-bound-lm} hold. Let \(1<m\le 2\) and \(\delta>0\). Assume that $1<m \le \frac{2(N-1)}{N}$ and $u_{\eps}\in L^\infty([0,T]; L^{p'+1}_{loc}(\R^N))$ for some $p'+1 > \frac{N(2-m)}{2}$ then for all $p>p'$,
    \begin{align}\label{I4'}
    \int_{\R^N} u_\varepsilon^{p+2}\psi
    \lesssim
    \delta\sup_{x_0\in\R^N}\int_{\R^N}\left|\nabla u_\varepsilon^{\frac{p+m}{2}}\right|^2\psi(x-x_0)
    + C_\delta \kappa^{-N}.
    \end{align}
Also, \eqref{I4'} hold same when \(
    \frac{2(N-1)}{N}<m\le 2,
    \)
     for any \(p>1\).
\end{lem}

\begin{proof}[Proof of Lemma \ref{uLp+2}.]
First we show that the local \(L^1\)-norm of \(u_\eps\) is bounded for \(1<m\le 2\). 

Multiplying the \(u_\eps\)-equation by \(\psi\) and integrating by parts, then applying Young's inequality, we obtain
\begin{align}\label{uL1}
\frac{d}{dt}\int_{\R^N} u_\eps\psi + {2\lambda} \int_{\R^N} u_\eps\psi
&\le \int_{\R^N} (u_\eps+\varepsilon)^m |\Delta \psi|
   + |\chi| \int_{\R^N} u_\eps|\nabla v_\eps||\nabla \psi|
   + \int_{\R^N}\big((a+ {2\lambda})u_\eps-bu_\eps^2\big)\psi \nonumber\\
&\le \kappa^2\int_{\R^N} u_\eps^m\psi
   +\varepsilon^m \kappa^2\int_{\R^N}\psi
   + |\chi|\kappa\int_{\R^N} \big(u_\eps^2\psi + |\nabla v_\eps|^2\psi\big)
   +\int_{\R^N}\big((a+ {2\lambda})u_\eps-bu_\eps^2\big)\psi \nonumber\\
&\lesssim C_\delta\kappa^{-N}
   + |\chi|\kappa\int_{\R^N} |\nabla v_\eps|^2\psi
   -\bigl(b-\kappa^2-|\chi|\kappa-\delta\bigr)\int_{\R^N}u_\eps^2\psi.
\end{align}
Multiplying through by \(e^{2\lambda t}\) and integrating in time yields
\begin{align}\label{ul1-psi}
\int_{\R^N} u_\eps(t)\psi
&\lesssim \int_{\R^N} u_0\psi + C_\delta\kappa^{-N}
+ |\chi|\kappa\iint_{\Omega_t} e^{-2\lambda (t-s)}|\nabla v_\eps|^2\psi \nonumber\\
&\qquad
-\bigl(b-\kappa^2-|\chi|\kappa-\delta\bigr)\iint_{\Omega_t} e^{-2\lambda(t-s)}u_\eps^2\psi.
\end{align}

When \(\tau=0\), by \eqref{est-tau0},
\begin{equation}\label{tau0-psi}
\int_{\R^N}|\nabla v_\eps|^2\psi \le C \int_{\R^N}u_\eps^2\psi.
\end{equation}

When \(\tau=1\), for any \(q>1\), let
\[
\psi_1:=\psi_{1,q}:=\psi^{1/q}.
\]
Then multiplying the \(v_\eps\)-equation by \(\psi_1\), and applying Lemma \ref{maximal-regularity-lm} then yields
\begin{align*}
&\int_{0}^t e^{-\lambda q(t-s)}
\Big(\|\Delta(v_\eps\psi_1)\|_{L^q(\R^N)}^{q}
+ \|\nabla(v_\eps\psi_1)\|_{L^q(\R^N)}^{q}
+\| v_\eps\psi_1\|_{L^q(\R^N)}^{q}\Big)\,ds \\
&\le C_{q,N}\int_{0}^t \int_{\R^N} e^{-\lambda q(t-s)}
\big|-2\nabla v_\eps\cdot \nabla\psi_1 - v_\eps\Delta\psi_1+\mu u_\eps \psi_1\big|^{q} \, dx\,ds
+ C_{q,N}\, t e^{-\lambda qt} \|v_0\psi_1\|_{W^{1,q}(\R^N)}^q \\
&\lesssim C_{q,N}\|v_0\|_{W^{1,\infty}(\R^N)}^q \|\psi\|_{L^1(\R^N)}
+ C_{q,N}\int_{0}^t \int_{\R^N} e^{-\lambda q(t-s)}
\big[ \kappa^q (|v_\eps\psi_1|^q+|\nabla (v_\eps \psi_1 )|^q) + \mu^q|u_\eps \psi_1|^q \big]\,dx\,ds.
\end{align*}
Here we used that, 
\(
|\nabla \psi_1|\lesssim \kappa \psi_1\) and \(
|\Delta \psi_1|\lesssim \kappa^2 \psi_1
\)
from Lemma \ref{psi-lm}. Therefore, if \(\kappa\) is sufficiently small depending only on \(q\) and \(N\), then
\begin{equation}\label{max-ine-psi}
\int_{0}^t e^{-\lambda q(t-s)}
\Big(\|\Delta(v_\eps\psi_1)\|_{L^q(\R^N)}^{q}
+ \|\nabla(v_\eps\psi_1)\|_{L^q(\R^N)}^{q}
+\| v_\eps\psi_1\|_{L^q(\R^N)}^{q}\Big)\,ds
\lesssim \kappa^{-N}+ \mu^q\int_{0}^t\int_{\R^N} e^{-\lambda q(t-s)} |u_\eps \psi_1|^q.
\end{equation}
Choosing \(q=2\), we can get that
\begin{equation}\label{tau1-psi}
\iint_{\Omega_t} e^{-2\lambda(t-s)}|\nabla v_\eps|^2\psi
\lesssim \kappa^{-N} + \mu^2\iint_{\Omega_t} e^{- 2\lambda(t-s)}u_\eps^2\psi.
\end{equation}

Finally, \eqref{tau0-psi} and \eqref{tau1-psi}, together with \eqref{ul1-psi}, imply that after choosing \(\delta\) and \(\kappa\) sufficiently small, there exists \(C>0\), independent of \(T\) and \(\varepsilon\), such that for all \(t\in[0,T]\),
\begin{equation}\label{uL1-psi}
\int_{\R^N} u_\eps(t)\psi \le C.
\end{equation}
Next by the properties of \(\psi\) in Lemma \ref{psi-lm}, and a change of variable,
\begin{align}\label{up+2}
    \int_{\R^N} u_\eps^{p+2}\psi \lesssim \sum_{\k z\in \Z^N}\psi(z)\int_{B_{N/\kappa}(z)} u_\eps^{p+2}(\cdot, x)\, dx =\sum_{\k z\in \Z^N}\psi(z)\Big(\frac{N}{\k}\Big)^N\int_{B_{1}} u_\eps^{p+2}(\cdot, z+yN/\kappa)\, dy .
\end{align}
Let
\[
w:=u_\eps^{\frac{p+m}{2}},\qquad 
r:=\frac{2(p+2)}{p+m}.
\]
By Gagliardo-Nirenberg inequality (see \cite{nirenberg1966extended} and Lemma 2.4 in \cite{wang2014quasilinear}), and then a change a variable again, there exists \(C>0\) independent of \(\k\) such that
\begin{align*}
\int_{B_{N/\kappa}} u_\eps^{p+2}(\cdot, x)\, dx = \|w\|^r_{L^r(B_{N/\kappa})}
&\le C\,\|\nabla w\|_{L^2(B_{N/\kappa})}^{\theta r}\,\|w\|_{L^q(B_{N/\kappa})}^{(1-\theta)r}
+ C\,\|w\|^r_{L^q(B_{N/\kappa})},
\end{align*}
where 
\[
\theta=\frac{
\frac{1}{q}-\frac{1}{r}
}{
\frac1N-\frac 12+\frac{1}{q}
}.
\]
Now suppose \(1<m \le \frac{2(N-1)}{N}\) and $u_{\eps}\in L^\infty([0,T]; L^{p'+1}_{loc}(\R^N))$ for some \(p'+1 > \frac{N(2-m)}{2}\), take \(q:=\frac{2(p'+1)}{p+m}\). Then for any $p>p'$
\[
\theta=\frac{
\frac{(m+p)N}{2(p'+1)}-\frac{(m+p)N}{2(p+2)}
}{
1-\frac N2+\frac{(m+p)N}{2(p'+1)}
}
=
\frac{N(m+p)(p-p'+1)}{(p+2)(N(m+p)+(p'+1)(2-N))},
\]
and 
\[
\theta r=\frac{2N(p-p'+1)}{N(p+m)-(N-2)(p'+1)} \in(0, 2).
\]
Since \(\|w\|_{L^q(B_{N/\kappa})} = \|u_\eps\|_{L^{p'+1}(B_{N/\kappa})}\lesssim 1\), then by Young's inequality,
\begin{align}\label{lem3.2in1}
\int_{B_{N/\kappa}} u_\eps^{p+2} 
&\lesssim \|\nabla w\|_{L^2(B_{N/\kappa})}^{\theta r} + 1\nonumber\\
&\lesssim \delta\|\nabla w\|^2_{L^2(B_{N/\kappa})} + C_\delta.
\end{align}
Similarly, when \(\frac{2(N-1)}{N} <m\le 2\), take \(q =\frac{2}{p+m}\), then \(\|w\|_{L^q(B_{N/\kappa})} = \|u_\eps\|_{L^1(B_{N/\kappa})}\lesssim 1\), by \eqref{uL1-psi}. In this case,
\[
\theta=\frac{N(p+m)(p+1)}{(p+2)\big(N(p+m)-(N-2)\big)} \implies \theta r=\frac{2N(p+1)}{N(p+m)-(N-2)}.
\]
Note that \(\theta r \in(0, 2)\) whenever \(m>\frac{2(N-1)}{N}\). Thus by Young's inequality again, equation \eqref{lem3.2in1} holds as well. Now substituting \eqref{lem3.2in1} into \eqref{up+2} with the fact that $\psi(x-z) =1$ for $x\in B_{N/\k}$, gives
\begin{align*}
    \int_{\R^N} u_\eps^{p+2}\psi^2 &\le \sum_{\k z\in \Z^N}\psi(z)\left[\delta\int_{B_{N/\kappa}} |\nabla u_\eps^{\frac{p+m}{2}}|\psi^2(x-z) + C_\delta\right]\nonumber\\
    &\lesssim \delta\sup_{x_0\in\R^N}\int_{\R^N} |\nabla u_\eps^{\frac{p+m}{2}}|^2\psi^2(x-x_0) + C_\delta\kappa^{-N}.
\end{align*}
This completes the proof of the lemma.
\end{proof}

The next lemma provides an estimate involving \(\nabla v_{\eps}\) that will be useful in controlling the \(L^p\)-bound of the solution in the parabolic--parabolic case.

\begin{lem}\label{l4.7}
Suppose that the assumptions of Lemma \ref{v-bound-lm} hold with \(\tau =1\). Let \(r>1\). Then
\begin{align}\label{nablavr}
        &\frac{1}{r}\frac{d}{dt}\int_{\R^N}|\nabla v_\eps|^{2r}\psi + C\int_{\R^N}|\nabla v_\eps|^{2r-4}|\nabla |\nabla v_\eps|^2|^2\psi + C\int_{\R^N}|\nabla v_\eps|^{2r}\psi+ \int_{\R^N}|D^2v_\eps|^2|\nabla v_\eps|^{2r-2}\psi\nonumber\\
       & \qquad\qquad \le C \int_{\R^N}u_\eps^2|\nabla v_\eps|^{2r-2}\psi.
\end{align}
Also, for \(1<m\le 2\), we have
\[
\nabla v_\eps\in L^\infty([0,T]; L^2_{loc}(\R^N)).
\]
\end{lem}

\begin{proof}
The constants \(c\) and \(C\) below are generic, but independent of \(\eps\), \(\kappa\), and \(T\).

From the second equation in \eqref{main-perturbed-eq}, we have
\begin{align*}
    (|\nabla v_\eps|^2)_t
    =2\nabla v_\eps\cdot\nabla (v_\eps)_t
    &=2\nabla v_\eps\cdot \nabla \Delta v_\eps-2\lambda|\nabla v_\eps|^2 + 2\mu\nabla v_\eps\cdot \nabla u_\eps\\
    &=\Delta |\nabla v_\eps|^2-2|D^2v_\eps|^2-2\lambda|\nabla v_\eps|^2 + 2\mu\nabla v_\eps\cdot \nabla u_\eps.
\end{align*}
Upon multiplication by \((|\nabla v_\eps|^2)^{r-1}\psi\) and integration, this leads to
\begin{align}\label{vest1}
        &\frac{1}{r}\frac{d}{dt}\int_{\R^N}|\nabla v_\eps|^{2r}\psi\nonumber\\
        &\qquad= \int_{\R^N}\Delta |\nabla v_\eps|^2|\nabla v_\eps|^{2r-2}\psi-2\int_{\R^N}|D^2v_\eps|^2|\nabla v_\eps|^{2r-2}\psi-2\lambda\int_{\R^N}|\nabla v_\eps|^{2r}\psi \nonumber\\
        &\qquad\quad + 2\mu\int_{\R^N}|\nabla v_\eps|^{2r-2}\psi\nabla v_\eps\cdot \nabla u_\eps\nonumber\\
        &\qquad= -(r-1)\int_{\R^N}|\nabla v_\eps|^{2r-4}|\nabla |\nabla v_\eps|^2|^2\psi-\int_{\R^N}|\nabla v_\eps|^{2r-2}\nabla|\nabla v_\eps|^2\cdot\nabla\psi \nonumber\\
        &\qquad\quad -2\int_{\R^N}|D^2v_\eps|^2|\nabla v_\eps|^{2r-2}\psi-2\lambda\int_{\R^N}|\nabla v_\eps|^{2r}\psi - 2\mu(r-1)\int_{\R^N}u_\eps |\nabla v_\eps|^{2r-4}\nabla|\nabla v_\eps|^2\cdot\nabla v_\eps\psi \nonumber\\
        &\qquad\quad - 2\mu\int_{\R^N} u_\eps |\nabla v_\eps|^{2r-2}\Delta v_\eps\psi
        - 2\mu\int_{\R^N} u_\eps |\nabla v_\eps|^{2r-2}\nabla v_\eps\cdot \nabla \psi.
\end{align}
Then by Young's inequality, we have the following estimates:
\[
|\nabla v_\eps|^{2r-2}\nabla|\nabla v_\eps|^2\cdot\nabla\psi
\le \frac{\kappa}{2}|\nabla|\nabla v_\eps|^2|^2|\nabla v_\eps|^{2r-4}\psi + \frac{\kappa}{2}|\nabla v_\eps|^{2r}\psi,
\]
\[
2\mu(r-1)u_\eps|\nabla v_\eps|^{2r-4}\nabla|\nabla v_\eps|^2\cdot\nabla v_\eps\psi
\le\frac{r-1}{8}|\nabla|\nabla v_\eps|^2|^2|\nabla v_\eps|^{2r-4}\psi + 8(r-1)\mu^2u_\eps^2|\nabla v_\eps|^{2r-2}\psi,
\]
\[
2\mu u_\eps|\nabla v_\eps|^{2r-2}\Delta v_\eps\psi
\le |D^2v_\eps|^2|\nabla v_\eps|^{2r-2}\psi + N\mu^2 u_\eps^2|\nabla v_\eps|^{2r-2}\psi,
\]
\[
2\mu u_\eps|\nabla v_\eps|^{2r-2}\nabla v_\eps\cdot \nabla \psi
\le \frac{\kappa}{2}|\nabla v_\eps|^{2r}\psi + \frac{\kappa}{2}\mu^2 u_\eps^2|\nabla v_\eps|^{2r-2}\psi.
\]
Plugging these inequalities into \eqref{vest1}, we obtain, for \(\kappa\) sufficiently small,
\begin{align*}
        &\frac{1}{r}\frac{d}{dt}\int_{\R^N}|\nabla v_\eps|^{2r}\psi + C\int_{\R^N}|\nabla v_\eps|^{2r-4}|\nabla |\nabla v_\eps|^2|^2\psi + C\int_{\R^N}|\nabla v_\eps|^{2r}\psi+ \int_{\R^N}|D^2v_\eps|^2|\nabla v_\eps|^{2r-2}\psi\\
       & \qquad\qquad \le C \int_{\R^N}u_\eps^2|\nabla v_\eps|^{2r-2}\psi.
\end{align*}

Now we establish the second part of the lemma. Multiplying the \(v_\eps\)-equation by \(-2\Delta v_\eps\psi\), and then integrating both sides, gives
\begin{align*}
    \frac{d}{dt}\int_{\R^N}|\nabla v_\eps|^2\psi
    &= -2\int_{\R^N}\Delta v_\eps(\Delta v_\eps -\lambda v_\eps + \mu u_\eps)\psi\\
    & = -2\int_{\R^N}|\Delta v_\eps|^2\psi -2\lambda\int_{\R^N}|\nabla v_\eps|^2\psi-2\lambda \int_{\R^N}v_\eps\nabla v_\eps\cdot \nabla \psi-2\mu \int_{\R^N} \Delta v_\eps u_\eps \psi \\
    &\le -\int_{\R^N}|\Delta v_\eps|^2\psi -\lambda\int_{\R^N}|\nabla v_\eps|^2\psi + \mu^2\int_{\R^N} u_\eps^2\psi+\kappa\lambda\int_{\R^N} v_\eps^2\psi.
\end{align*}
Also, multiplying the \(v_\eps\)-equation by \(2v_\eps\psi\) and then integrating gives
\begin{align*}
    \frac{d}{dt}\int_{\R^N}v_\eps^2\psi
    &= \int_{\R^N} 2v_\eps(\Delta v_\eps -\lambda v_\eps + \mu u_\eps)\psi\\
    & = -2\int_{\R^N}|\nabla v_\eps|^2\psi -2\int_{\R^N} v_\eps\nabla v_\eps\cdot\nabla \psi -2\lambda\int_{\R^N}v_\eps^2\psi+2\mu \int_{\R^N} v_\eps u_\eps \psi \\
    &\le -\int_{\R^N}|\nabla v_\eps|^2\psi -(\lambda-\kappa)\int_{\R^N}v_\eps^2\psi + \frac{\mu^2}{\lambda}\int_{\R^N} u_\eps^2\psi.
\end{align*}
Combining these two estimates gives
\begin{align}\label{vnablav-est}
    &\frac{d}{dt}\left(\int_{\R^N}|\nabla v_\eps|^2\psi + \int_{\R^N}|v_\eps|^2\psi \right)+ \int_{\R^N}|\nabla v_\eps|^2\psi + \int_{\R^N}|v_\eps|^2\psi\lesssim  \kappa\int_{\R^N}|v_\eps|^2\psi +c\int_{\R^N}u_\eps^2\psi.
\end{align}
Recall that from \eqref{uL1}, we have that for \(1<m\le 2\),
\begin{align*}
    \frac{d}{dt}\int_{\R^N} u_\eps\psi^2 + \int_{\R^N} u_\eps\psi^2 \lesssim C_\delta\kappa^{-N} + \kappa\int_{\R^N}  |\nabla v_\eps|^2\psi^2 -c\int_{\R^N}u_\eps^2\psi^2.
\end{align*}
Adding \eqref{vnablav-est} to this inequality then implies that for \(1<m\le 2\),
\begin{align*}
    &\frac{d}{dt}\left(\int_{\R^N}|\nabla v_\eps|^2\psi + \int_{\R^N}|v_\eps|^2\psi +  \int_{\R^N} u_\eps\psi\right) +  \int_{\R^N}|\nabla v_\eps|^2\psi +\int_{\R^N}|v_\eps|^2\psi+ \int_{\R^N} u_\eps\psi\nonumber\\
    &\lesssim  C_\delta \kappa^{-N} + \kappa \int_{\R^N}|v_\eps|^2\psi + \kappa\int_{\R^N}  |\nabla v_\eps|^2\psi -c\int_{\R^N}u_\eps^2\psi.
\end{align*}
Choosing \(\kappa\) small enough, we have for some \(C>0\),
\begin{align}
    \frac{d}{dt}\left(\int_{\R^N}|\nabla v_\eps|^2\psi + \int_{\R^N}|v_\eps|^2\psi +  \int_{\R^N} u_\eps\psi\right) +  \int_{\R^N}|\nabla v_\eps|^2\psi +\int_{\R^N}|v_\eps|^2\psi+ \int_{\R^N} u_\eps\psi\lesssim  C.
\end{align}
Hence
\[
u_\eps\in L^1_{loc}(\R^N),\qquad v_\eps\in L^2_{loc}(\R^N),\qquad \nabla v_\eps\in L^2_{loc}(\R^N)
\]
for \(1<m\le 2\) and for all \(t\).
\end{proof}

\section{Uniform Local $L^{p+1}$ Estimates}
In this section, we prove a uniform local \(L^p\) bound for solutions of the perturbed problem, independent of \(T\) and \(\varepsilon\). The lemmas established in this section will be used to prove Theorems 1.1 and 1.2.

\begin{lem}[Local $L^{p}$ estimate]\label{Lp-prior}
\label{apriori-prop}
Suppose that the assumptions of Lemma \ref{v-bound-lm} hold. For \(p>1\), assume that
\begin{equation*}
    b> b_{m, p} := \begin{cases}
        \left(\frac{p-1}{p}\left(C_{p+1,N}\right)^{\frac{1}{p+1}}\right)|\chi|\mu, \quad & 1<m\le \frac{2N}{N+2},\\[1mm]
        0, & m> \frac{2N}{N+2},
    \end{cases}
\end{equation*}
when \(\tau=1\), and
\begin{equation*}
    b> b_{m, p} := \begin{cases}
        \frac{|\chi|\mu (p-1)}{p}, \quad & 1<m\le 2-\frac2N,\\[1mm]
        0, & m> 2-\frac2N,
    \end{cases}
\end{equation*}
when \(\tau = 0\). Then, for \(t\in[0,T]\), $p>1$,
\begin{equation*}
u_\eps(t,\cdot)\in L_{\rm loc}^{p}(\R^N),
\end{equation*}
with a bound depending only on \(m,|\chi|,a,b,N,p,\|u_0\|_\infty,\lambda,\mu,\|\tau v_0\|_{W^{1,\infty}}\), and the diameter of the local spatial domain, but independent of \(T\) and \(\eps\). Here \(C_{p+1,N}\) is the constant appearing in the maximal regularity lemma, Lemma \ref{maximal-regularity-lm}.
\end{lem}
As a corollary of the above local $L^{p}$ estimates, we have the following on the existence of a globally defined bounded non-negative classical  solutions to \eqref{main-perturbed-eq}. 
\begin{prop}[Classical solutions of \eqref{main-perturbed-eq}]
\label{main-perturbed-thm}
Assume that $m>1$, $a, b, \lambda, \mu>0$, $\chi\in\R$  $\tau\in \{0, 1\}$, $u_{0}, \tau v_{0}\ge 0$ such that $u_0$ is uniformly $C^{1+\alpha}$, and $\tau v_0$ is uniformly $C^{2+\alpha}$. For any $b>0$ satisfying  

\begin{equation*}
    b> b_m := \begin{cases}
        \left(\inf_{{p} >\max\{1,N(2-m)/2m\}}\frac{p-1}{p}\left(C_{p+1,N}\right)^{\frac{1}{p+1}}\right)|\chi|\mu, \quad & 1<m\le \frac{2N}{N+2},\\[1mm]
        0, & m> \frac{2N}{N+2},
    \end{cases}
\end{equation*}
when \(\tau=1\), and
\begin{equation*}
    b> b_m := \begin{cases}
        \frac{\mu|\chi|((2-m)N -2)_+}{(2-m)N}, \quad & 1<m\le 2-\frac2N,\\[1mm]
        0, & m> 2-\frac2N,
    \end{cases}
\end{equation*}
 
 when \(\tau=0\), there exists a  unique non-negative globally bounded classical solution  $(u_\eps$,  $v_\eps)$ of \eqref{main-perturbed-eq}  with initial condition 
$u_0,v_0$ for $\tau =1$ and $u_0$ when $\tau =0$. 
\end{prop}

In the next two subsections, we shall prove Lemma 3.1 for the parabolic-parabolic case and the parabolic-elliptic case, respectively. Subsequently, we shall drop $\eps$ when we refer to the solution of \eqref{main-perturbed-eq} for convenience.

\subsection{Proof of Lemma \ref{Lp-prior} when $\tau =1$.}\label{Lploc-est1}

In this subsection we establish the local $L^{p}$ norm estimate for non-negative solution of the perturbed problem, when $\tau =1$ on any finite time interval $[0,T]$. Let $0<\eps<1$, $0<\kappa<1$, and take $p>1$.  Recall that
$$\Omega_t:=[0,t]\times\R^N\quad \text{for any} \quad t\in (0,T].$$ 
Let $\psi$ be from Lemma \ref{psi-lm} with parameter $\kappa$. For convenience, we would establish $L^{p+1}$ estimate. Multiplying the first equation of \eqref{main-perturbed-eq} by $u^{p}\,\psi^2$ and integrating over $\R^N$  yields

\begin{align}
\label{new-new-lp-eq1}
&\quad\, \frac{1}{p+1}\frac{d}{dt}\int_{\R^N} u^{p+1}\psi^2  dx\nonumber \\
&= \int_{\R^N} u^{p}\psi^2\nabla\cdot\left[m(u+\eps)^{m-1}\nabla u - \chi u\nabla v\right] + u^{p}\psi^2f(u)\nonumber \\
&\leq  -mp\int_{\R^N} u^{p+m-2}|\nabla u|^2\psi^2 +m\int_{\R^N}u^{p}(u+\eps)^{m-1}|\nabla u||\nabla\psi^2| \nonumber\\
 &\quad+ \chi p \int_{\R^N} u^{p} \nabla u\cdot\nabla v\,\psi^2+ |\chi|\int_{\R^N} u^{p+1} |\nabla v||\nabla \psi^2|+\int_{\R^N}u^{p}(au - bu^2)\psi^2\nonumber\\
&\leq \underbrace{-mp\int_{\R^N} u^{p+m-2}|\nabla u|^2\psi^2}_{-I_1} +\underbrace{C\kappa\int_{\R^N}u^{p+m-1}|\nabla u|\psi^2}_{\kappa I_2}+ \underbrace{C\kappa\eps^{m-1}\int_{\R^N}u^{p}|\nabla u|\psi^2}_{\kappa I_3}\nonumber \\
 &\quad+ \underbrace{|\chi|\kappa\int_{\R^N} u^{p+1}|\nabla v| \psi^2}_{\kappa I_4} +\underbrace{ \int_{\R^N}  (au^{p+1}-b u^{p+2})\psi^2}_{I_5}+\underbrace{ \chi p \int_{\R^N} u^{p} \nabla u\cdot\nabla v\,\psi^2}_{I_6},
\end{align}
We now estimate each term in the above inequality. First, by Young's inequality,
\begin{align}
\label{I2-eq1}
\kappa I_2=  \kappa\int_{\R^N}u^{p+m-1}|\nabla u|\psi^2
&\le \frac{mp}{4}\int_{\R^N} u^{p+m-2}|\nabla u|^2 \psi^2 + \frac{m}{p}\kappa^2\int_{\R^N}u^{p+m}\psi^2.
\end{align}
Note that, when \(m>2\), the second term cannot be controlled by the logistic source. To handle this, we use the continuity argument introduced in \cite{hassan2025global}. 
For the remaining terms, the main difference from the model considered in \cite{hassan2025global} lies in the treatment of the terms involving \(v\), as mentioned in Section 1, in particular the term \(I_6\). We now proceed with the proof of the lemma.

\begin{proof} [Proof of Lemma \ref{Lp-prior} for $\tau =1$:]

\medskip

\noindent{\bf Case 1: When $1<m\le 2$.}

\medskip
Using Young's inequality, we have for $0<\delta<1$,
\begin{align}
\label{I2-eq1'}
\kappa I_2
&\le \frac{mp}{4}\int_{\R^N} u^{p+m-2} |\nabla u|^2 \psi^2dx + \delta \int_{\R^N}u^{p+2}\psi^2 dx+C_\delta\kappa^{-N},
\end{align}
where in the last inequality, we used the properties of $\psi $ in Lemma \ref{psi-lm}. Similarly by Young's inequality again, we have the following
\begin{align}\label{I6-eq1}
   \kappa I_3(t)
&\leq \frac{mp}{4}\int_{\R^N} u^{p+m-2}|\nabla u|^2\psi^2 +\delta\int_{\R^N}u^{p+2}\psi^2+ C_\delta\kappa^{-N}\nonumber\\
   \kappa I_4(t)&\le \k|\chi| \int_{\R^N} u^{p+2}\psi^2 +\kappa|\chi|A_p\int_{\R^N} |\nabla v|^{p+2}\psi^2\nonumber\\
   I_5&\le -\int_{\R^N} u^{p+1}\psi^2dx -(b-\delta)\int_{\R^N} u^{p+2}\psi^2 dx+C_\delta\kappa^{-N}
\end{align}
Where $A_p := \frac{1}{p+2}\left(\frac{p+2}{p+1}\right)^{-(p+1)}$. Now we would estimate $I_6$ for different cases also.

\medskip

\noindent{\bf Case 1.1: When $1 < m\le \frac{2N}{N+2}$}

For $r>0$, we have
\begin{align}\label{I4-eq1}
   I_6 &=  \chi p \int_{\R^N} u^{p} \nabla u\cdot\nabla v\,\psi^2    = \frac{ \chi p}{p+1} \int_{\R^N} \nabla(u^{p+1}) \cdot\nabla v\,\psi^2 \nonumber\\
   & \le \frac{ |\chi| p}{p+1}\int_{\R^N} u^{p+1} |\Delta v|\,\psi^2 + \frac{ 2\kappa|\chi| p}{p+1}\int_{\R^N} u^{p+1} |\nabla v|\psi^2\nonumber\\
    & \le (r+ 2\kappa|\chi|)\int_{\R^N}u^{p+2}\psi^2 + \left(\frac{p}{p+1}\right)^{p+2}A_{p}\left[|\chi|^{p+2}r^{-(p+1)}\int_{\R^N} |\Delta v|^{p+2}\psi^2 + 2\kappa|\chi| \int_{\R^N}|\nabla v|^{p+2}\psi^2\right],
\end{align}
Combining inequalities \eqref{new-new-lp-eq1} -\eqref{I4-eq1} together gives
\begin{align*}
&\frac{1}{p+1}\frac{d}{dt}\int_{\R^N} u^{p+1}\psi^2  dx +
\int_{\R^N} u^{p+1}\psi^2 \le -\frac{mp}{2} \int_{\R^N}  u^{p+m-2} |\nabla u|^2\psi^2 -(b-3\delta- 3|\chi|\k)\int_{\R^N} u^{p+2}\psi^2 \nonumber\\
&\quad + \left(\frac{|\chi|p}{p+1}\right)^{p+2}A_{p}r^{-(p+1)}\int_{\R^N} |\Delta v|^{p+2}\psi^2 + \kappa |\chi| A_p\left((\frac{p}{p+1})^{p+2} +2\right)\int_{\R^N} |\nabla v|^{p+2}\psi^2+C_\delta \kappa^{-N} 
\end{align*}
Multiplying $e^{\lambda(p+1)t}$ to both sides and integrating in time, this implies that
\begin{align}
\label{aux-lp'}
&\int_{\R^N}u^{p+1}(t,x)\psi^2(x)dx+\frac{mp}{2}\iint_{\Omega_{t}} e^{-\lambda(p+1)(t-s)}u^{p+m-2}|\nabla u|^2\psi^2dxds  \nonumber\\
&\lesssim -(b-3\delta - 3|\chi|\k -r)\iint_{\Omega_{t}}e^{-\lambda(p+1)(t-s)}u^{p+2}\psi^2dxds\nonumber + \left(\frac{|\chi|p}{p+1}\right)^{p+2}A_{p}r^{-(p+1)}\iint_{\Omega_t} e^{-\lambda(p+1)(t-s)}|\Delta v|^{p+2}\psi^2 \nonumber\\
&\quad + 3\kappa |\chi| \iint_{\Omega_t} e^{-\lambda(p+1)(t-s)}|\nabla v|^{p+2}\psi^2+\int_{\R^N} u^{p+1}(0,x)\psi^2dx 
+C_\delta \kappa^{-N} 
\end{align}
where $C_\delta$ is independent of $t$, we used the fact that $A_p, \left(\frac{p}{p+1}\right)^{p+2} \le 1$ in the last inequality. From \eqref{max-ine-psi}, we have
\begin{align*}
    &\left(\frac{|\chi|p}{p+1}\right)^{p+2}A_{p}r^{-(p+1)}\iint_{\Omega_t} e^{-\lambda(p+1)(t-s)} |\Delta v|^{p+2}\psi^2 + 3\kappa |\chi| \iint_{\Omega_t} e^{-\lambda(p+1)(t-s)}|\nabla v|^{p+2}\psi^2 \\
    &\lesssim \left[\left(\frac{|\chi|p}{p+1}\right)^{p+2}A_{p}r^{-(p+1)} + 3|\chi|\kappa\right]\left[\kappa^{-N} + \mu^{p+2}C_{p+2, N} \int_{0}^t\int_{\R^N}  e^{-\lambda(p+1)(t-s)} u^{p+2} \psi^2 \right]
\end{align*} 
Hence, choosing $\kappa\lesssim\delta$ in  \eqref{aux-lp'}, together with the above inequality, and regularity of our initial $u_0$, we get
\begin{align}\label{eq:u}
&\int_{\R^N}u^{p+1}(t,x)\psi^2(x)dx+\frac{mp}{2}\iint_{\Omega_{t}} e^{-\lambda(p+1)(t-s)}u^{p+m-2}|\nabla u|^2\psi^2dxds \nonumber\\
&\lesssim -(b-9\delta - r-r^{-(p+1)} \left(\frac{\mu|\chi|p}{p+1}\right)^{p+2}A_pC_{p+2,N })\iint_{\Omega_{t}}e^{-\lambda (p+1)(t-s)}u^{p+2}\psi^2dxds+ C_\delta \kappa^{-N} 
\end{align}

Note that
$$
\min_{r>0} \left(r+r^{-(p+1)} \left(\frac{\mu|\chi|p}{p+1}\right)^{p+2}A_pC_{p+2,N }\right)=\frac{p}{p+1}\left(C_{p+2,N}\right)^{\frac{1}{p+2}}|\chi|\mu.
$$
By the assumption $b> \frac{q-1}{q}\left(C_{q+1,N}\right)^{\frac{1}{q+1}}|\chi|\mu =\frac{p}{p+1}\left(C_{p+2,N}\right)^{\frac{1}{p+2}}|\chi|\mu$, (if we let $q= p+1$). Hence, we can choose $r>0$, $\delta>0$ small enough,  such that
$$
b\ge 9\delta +r+r^{-(p+1)} \left(\frac{\mu|\chi|p}{p+1}\right)^{p+2}A_{p}C_{p+2,N}.$$ 
Equation \eqref{eq:u} together with this inequality gives that $u\in L^\infty(0,T;  L^{p}_{loc}(\R^N))$ for any $p >1$ and $b>b_{m,p}$. This completes the proof of this case.

\medskip

\noindent{\bf Case 1.2: When $\frac{2N}{N+2} < m\le 2$.} 

In this case, by Young's inequality again, we have
\begin{align}\label{I4-eq2}
   I_6 &\le \frac{mp}{4}\int_{\R^N} u^{p+m-2}|\nabla u|^2 \psi^2 + \frac{p}{4m} \int_{\R^N} u^{p-m+2}|\nabla v|^2 \psi^2 \nonumber\\
   &\le \frac{mp}{4}\int_{\R^N} u^{p+m-2}|\nabla u|^2 \psi^2 + \delta \int_{\R^N} u^{p+2} \psi^2 + C_\delta\int_{\R^N} |\nabla v|^{\frac{2(p+2)}{m}} \psi^2 
\end{align}
Combining this inequality with \eqref{new-new-lp-eq1}-\eqref{I6-eq1} gives 
\begin{align*}
    & \frac{1}{p+1}\frac{d}{dt}\int_{\R^N} u^{p+1}\psi^2  dx + \int_{\R^N} u^{p+1}\psi^2 \le -\frac{mp}{4}\int_{\R^N} u^{p+m-2}|\nabla u|^2 \psi^2 -(c-3\delta+\k) \int_{\R^N} u^{p+2} \psi^2 \nonumber\\ 
    &\qquad\qquad + C_\delta\int_{\R^N} |\nabla v|^{\frac{2(p+2)}{m}} \psi^2 + \k\int_{\R^N} |\nabla v|^{p+2} \psi^2 + C_\delta \k^{-N} 
\end{align*}
Now summing the above inequality with \eqref{nablavr} from Lemma \ref{l4.7} gives that for any $q>1$,
\begin{align}\label{case1.2}
    & \frac{1}{p+1}\frac{d}{dt}\int_{\R^N} u^{p+1}\psi^2  dx + \frac{1}{q}\frac{d}{dt}\int_{\R^N}|\nabla v|^{2q}\psi^2 + C\int_{\R^N}|\nabla v|^{2q-4}|\nabla |\nabla v|^2|^2\psi^2 + C\int_{\R^N}|\nabla v|^{2q}\psi^2 + \int_{\R^N} u^{p+1}\psi^2\nonumber\\
       & \qquad\qquad \le -c\int_{\R^N} u^{p+m-2}|\nabla u|^2 \psi^2 -(c-3\delta+\k) \int_{\R^N} u^{p+2} \psi^2  + C_\delta\int_{\R^N} |\nabla v|^{\frac{2(p+2)}{m}} \psi^2 + C \int_{\R^N}u^2|\nabla v|^{2q-2}\psi^2 \nonumber\\
       & \qquad\qquad + \k\int_{\R^N} |\nabla v|^{p+2} \psi^2 + C_\delta \k^{-N} \nonumber\\
       & \qquad\qquad \lesssim -c\int_{\R^N} u^{p+m-2}|\nabla u|^2 \psi^2 -(c-4\delta+\k) \int_{\R^N} u^{p+2} \psi^2  + C_\delta\int_{\R^N} |\nabla v|^{\frac{2(p+2)}{m}} \psi^2 \nonumber\\
       & \qquad\qquad + C_\delta \int_{\R^N} |\nabla v|^{\frac{2(p+2)(q-1)}{p}}\psi^2 + C_\delta \k^{-N}. 
\end{align}
The last inequality holds by Young's inequality.
Now by a change of variable and Gagliardo Sobolev inequality, we can get a $C$ independent of $\k$, such that 
\begin{align*}
    \int_{\R^N} |\nabla v|^{a}\psi^2 &\le \sum_{kz\in\Z^N} \psi^2(z) \int_{B_{N/\k}} |\nabla v|^{a}(x) \, dx 
    = \sum_{kz\in\Z^N} \psi^2(z)\||\nabla v|^q\|_{L^{\frac{a}{q}}(B_{N/\k})}^{\frac{a}{q}}\\
    &\le \sum_{kz\in\Z^N} C\psi^2(z)\left[\|\nabla|\nabla v|^q\|^{\theta a/q}_{L^2(B_{N/\k})}\||\nabla v|^q\|^{(1-\theta)a/q}_{L^\frac{2}{q}(B_{N/\k})} + \||\nabla v|^q\|_{L^\frac{2}{q}(B_{N/\k})}\right]\\
    &\lesssim \sum_{kz\in\Z^N} \psi^2(z)\left[\|\nabla|\nabla v|^q\|^{\theta a/q}_{L^2(B_{N/\k})} + 1\right].
\end{align*}
The last inequality follows from the second part of Lemma \ref{l4.7} which gives that for $1<m\le 2$, $\nabla v\in L^{\infty}([0,T]; L^2_{loc}(\R^N))$. Also, 
\[\theta = \frac{\frac q2-\frac q{a}}
{\frac q2-\frac12+\frac1N}.\]
Now for $a = \frac{2(q-1)(p+2)}{p} >2$ when we choose $q>\frac{2(p+1)}{p+2}$, we have $\theta >0$ and 
\[
\theta a/q<2
\iff
\frac{\frac{(q-1)(p+1)}{p-1}-1}{\frac{q-1}{2}+\frac1N}<2
\iff
\frac{(q-1)(p+1)}{p-1}-1<q-1+\frac2N \iff q<\frac{p+1}{2}+\frac{p-1}{N}.
\]
Hence for $a = \frac{2(q-1)(p+2)}{p} >2$, 
\[
0<\theta a/q<2
\iff
\frac{2p}{p+1}<q<\frac{p+1}{2}+\frac{p-1}{N}.
\]
Similarly, for $a= \frac{2(p+2)}{m}$ 
\[
0<\theta a/q<2
\iff
p+1>m
\quad\text{and}\quad
q>\frac{p+1}{m}-\frac2N.
\]
Now for $p>\max\{2N-1, m-1, 1\}$, and $m>\frac{2N}{N+2}$, we can find a $$q\in \left(\max\{\frac{2p}{p+1}, \frac{p+1}{m}-\frac2N\}, \frac{p+1}{2}+\frac{p-1}{N}\right)$$
such that $0<\theta a/q <2$ for $a =  \frac{2(q-1)(p+2)}{p}$ and for $a= \frac{2(p+2)}{m}$. Hence for these $a's$  since $\psi(x-z) =1$ for $x\in B_{N/\k}(z)$, and by Young's inequality, we get
\begin{align*}
    \int_{\R^N} |\nabla v|^{a}\psi^2 
    &\lesssim \sum_{kz\in \Z^N} \psi^2(z)\left[\left(\int_{B_{N/\k}} |\nabla v|^{2q-4}|\nabla|\nabla v|^2|^{2}\psi^2(x-z)\right)^{\theta a/q} + 1\right]\\
    &\lesssim\delta \sup_{x_0\in\R^N} \int_{\R^N} |\nabla v|^{2q-4}|\nabla|\nabla v|^2|^{2}\psi^2(x-x_0)\, dx +C_{\delta}\k^{-N}
\end{align*}
Substituting this into \eqref{case1.2} gives that for $\delta$ and $\k$ small enough, $p>\max\{2N-1, m-1, 1\}$ and $\frac{2N}{N+2} <m \le 2$. 
\begin{align*}
    & \frac{d}{dt}\left(\int_{\R^N} u^{p+1}\psi^2  dx + \int_{\R^N}|\nabla v|^{2q}\psi^2\right) + \int_{\R^N}|\nabla v|^{2q}\psi^2 + \int_{\R^N} u^{p+1}\psi^2 + C\int_{\R^N}|\nabla v|^{2q-4}|\nabla |\nabla v|^2|^2\psi^2 \\  &\qquad\qquad \lesssim \delta \sup_{x_0\in\R^N} \int_{\R^N} |\nabla v|^{2q-4}|\nabla|\nabla v|^2|^{2}\psi^2(x-x_0)\, dx+ C_\delta \k^{-N}. 
\end{align*}
Multiplying through by $e^{\lambda(p+1)t}$ and integrating in time then shifting in space gives
\begin{align*}
    & \int_{\R^N} u^{p+1}\psi^2(x-x_0)  dx + C\iint_{\Omega_t} e^{-\lambda(p+1)(t-s)}|\nabla v|^{2q-4}|\nabla |\nabla v|^2|^2\psi^2(x-x_0) \\  &\qquad\qquad \lesssim \delta \sup_{x_0\in\R^N} \iint_{\Omega_t} e^{-\lambda(p+1)(t-s)} |\nabla v|^{2q-4}|\nabla|\nabla v|^2|^{2}\psi^2(x-x_0)\, dx+ C_\delta \k^{-N}. 
\end{align*}
Thus, taking supremum over $x_0\in \R^N$ and choosing $\delta$ small enough gives, $u(t, \cdot)\in L^{p}_{loc}(\R^N)$ and this completes the proof in this case.

\medskip

\noindent{\bf Case 2: When $m>2$.}

As stated earlier, we would use a continuity argument introduced in \cite{hassan2025global} in this case.

\medskip

Recall that from \eqref{I2-eq1'}, we have the following estimate for $I_2$
\[\kappa I_2
\le \frac{mp}{4}\int_{\R^N} u^{p+m-2}|\nabla u|^2 \psi^2 + \frac{m}{p}\kappa^2\int_{\R^N}u^{p+m}\psi^2\]
Now combining this inequality with \eqref{I6-eq1} for $I_3$-$I_5$ estimates and \eqref{I4-eq2} for $I_6$, gives 
\begin{align*}
    & \frac{1}{p+1}\frac{d}{dt}\int_{\R^N} u^{p+1}\psi^2  dx + \int_{\R^N} u^{p+1}\psi^2 + c\int_{\R^N} u^{p+m-2}|\nabla u|^2 \psi^2 \\
    &\lesssim \k^2\int_{\R^N} u^{p+m} \psi^2+ C \int_{\R^N} u^{p+2} \psi^2 + C_\delta\int_{\R^N} |\nabla v|^{\frac{2(p+2)}{m}} \psi^2 + + \k\int_{\R^N} |\nabla v|^{p+2} \psi^2 + C_\delta \k^{-N} \\
    &\lesssim \k^2\int_{\R^N} u^{p+m} \psi^2+ C \int_{\R^N} u^{p+2} \psi^2 +  \k\int_{\R^N} |\nabla v|^{p+2} \psi^2 + C_\delta \k^{-N}.
\end{align*}
We used Young's inequality and the fact that $m>2$, to get the last inequality. 
Multiplying $e^{\lambda(p+1)t}$ to both sides and integrating in time gives
\begin{align}
\label{aux-lp-eq1}
&\int_{\R^N}u^{p+1}(t,x)\psi^2(x)dx+c\iint_{\Omega_{t}} e^{-\lambda(p+1)(t-s)}u^{p+m-2}|\nabla u|^2\psi^2dxds\nonumber\\
&\lesssim \int_{\R^N} u^{p+1}(0,x)\psi^2dx + \kappa^2\iint_{\Omega_{t}} e^{-\lambda(p+1)(t-s)}u^{p+m}\psi^2
+\kappa\iint_{\Omega_{t}} e^{-\lambda(p+1)(t-s)}|\nabla v|^{p+2}\psi^2 +C_\delta \kappa^{-N}
\end{align}
By equation \eqref{max-ine-psi}, we have 
\begin{align}\label{aux-lp-eq2}
   \iint_{\Omega_{t}} e^{-\lambda(p+1)(t-s)}|\nabla v|^{p+2}\psi^2&\lesssim \kappa^{-N}+ \iint_{\Omega_t} 
e^{-\lambda(p+1)(t-s)} u^{p+2}\psi^2\nonumber\\
&\lesssim \kappa^{-N}+ \delta_1\iint_{\Omega_t} 
e^{-\lambda(p+1)(t-s)} u^{p+m}\psi^2
\end{align}
Substituting \eqref{aux-lp-eq2} into \eqref{aux-lp-eq1} and using the boundedness of the initial, we get that for $\delta_1 <\k$
\begin{align}
\label{aux-lp-eq3'}
&\int_{\R^N}u^{p+1}(t,x)\psi^2(x)dx+c\iint_{\Omega_{t}} e^{-\lambda(p+1)(t-s)}u^{p+m-2}|\nabla u|^2\psi^2dxds\nonumber\\
&\lesssim  \k^2\iint_{\Omega_{t}}e^{-\lambda(p+1)(t-s)}u^{p+m}\psi^2dxds + \kappa^{-N}
\end{align}

For any given $x_0\in\R^N$ and $q>1$, define
\begin{equation*}
X_{ q,x_0}(t)=\int_{\R^N} u^q(t, x) \psi^2(x-x_0)dx,\quad X_q(t)=\sup_{x_0\in\R^N} X_{q,x_0}(t),
\end{equation*}
\begin{equation*}
\text{and}\quad Y_{q}(t):=\sup_{s\in[0,t]}X_{q}(s).
\end{equation*}
Note that it is sufficient to prove the estimates for $N\ge 3$. By the properties of $\psi$,
\begin{align}\lb{I1-eq}
\int_{\R^N} u^{p+m-2}|\nabla u|^2\psi^2dx\gtrsim \int_{\R^N} |\nabla (u^{\frac{p+m}{2}})|^2\psi^2
\geq  c\int_{\R^N} |\nabla (u^{\frac{p+m}{2}}\psi)|^2
-C\kappa^2 \int_{\R^N} u^{p+m}\psi^2.
\end{align}
Applying Lemma \ref{I1-lm} with $r$ being replaced by $\frac{p+m}{2}$ (here we need $N\geq3$, the result  for when $N<3$ follows from the case $N\ge 3$), we have
\begin{align}\label{new-poin'}
\kappa^2\int_{\R^N}u^{p+m}\psi^2
&\leq  \delta\|\nabla(u^{\frac{p+m}{2}}\psi)\|_{2}^2+C_{\delta}\kappa^{2+{\theta^*}}\left[\int_{\R^N}u^{\frac{p+m}{2}+1}\psi^{q^*}dx\right]^{2/q^*}, 
\end{align}
where
$$
q^*:=\frac{p+m+2}{p+m}\in (1, 2),
\quad{and}\quad
{\theta^*}:=\frac{N(p+m-2)}{p+m+2}>0.
$$
Using the properties of $\psi$ from Lemma \ref{psi-lm}, together with Young's inequality in the above inequality, it follows that for $p\ge m$,
\begin{align*}
\int_{\R^N}u^{\frac{p+m}{2}+1}(s,x)\psi^{q^*}(x)\,dx
&\lesssim \sum_{\kappa z\in\Z^N}\psi^{q^*}(z)\int_{B_{N/\kappa}(z)}u^{\frac{p+m}{2}+1}(s,x)\psi^2(x-z)\,dx\\
&\lesssim \sup_{x_0\in\R^N} \int_{\R^N}u^{\frac{p+m}{2}+1}(s,x)\psi^2(x-x_0)\,dx\\
&\lesssim \sup_{x_0\in\R^N} \int_{\R^N}u^{p+1}(s,x)\psi^2(x-x_0)\,dx+\int_{\R^N}\psi^2(x-x_0)\,dx\\
&\lesssim X_{p+1}(s)+\kappa^{-N}.
\end{align*}
It follows from \eqref{new-poin'} that
\beq
\lb{new-poin}
\kappa^2\int_{\R^N}u^{p+m}\psi^2
\leq  \delta\|\nabla(u^{\frac{p+m}{2}}\psi)\|_{2}^2+C_{\delta}\kappa^{2+{\theta^*}}X_{p+1}(s)^{2/q^*}+C\kappa^{2+{\theta^*}-2N/q^*}.
\eeq
Note that $2+{\theta^*}-2N/q^* = 2-N$. Thus, taking $\delta$ small enough,
\eqref{I1-eq} and \eqref{new-poin} yields,
\begin{equation}
\label{I1-eq2}
\frac{c}{2} \int_{\R^N} u^{p+m-2}|\nabla u|^2\psi^2dx\geq c\|\nabla(u^{\frac{p+m}{2}}\psi)\|_{2}^2-C\kappa^{2+{\theta^*}}X_{p+1}(s)^{2/q^*}-C\kappa^{2-N}.
\end{equation}
Adding \eqref{I1-eq2} and \eqref{new-poin} gives that for $\delta<c/4$, and $p\ge m$,
\begin{equation}\label{lup+m}
   \kappa^2\int_{\R^N}u^{p+m}\psi^2
\lesssim  \frac{c}{2} \int_{\R^N} u^{p+m-2}|\nabla u|^2\psi^2dx + C_{\delta}\kappa^{2+{\theta^*}}X_{p+1}(s)^{2/q^*}+C\kappa^{2-N}. 
\end{equation}
Substituting this into \eqref{aux-lp-eq3'} we have that for $p\ge m$,
\begin{align*}
\int_{\R^N}u^{p+1}(t,x)\psi^2(x)dx+c/2\iint_{\Omega_{t}} e^{-\lambda(p+1)(t-s)}u^{p+m-2}|\nabla u|^2\psi^2dxds \lesssim    \kappa^{-N} +  C_{\delta}\kappa^{2+{\theta^*}}Y_{p+1}(t)^{2/q^*}.
\end{align*}
Shifting in space and taking supremum over $x_0$ in $\R^N$ gives
\begin{align*}
&\sup_{x_0\in\R^N}\int_{\R^N}u^{p+1}(t,x)\psi^2(x-x_0)dx\lesssim  \kappa^{-N}+  C_{\delta}\kappa^{2+{\theta^*}}Y_{p+1}(t)^{2/q^*}
\end{align*}
Taking supremum in $t\in [0,t_0]$ for any $t_0\in [0,T]$,  there exists $C_0>0$ independent of $t_0$ such that for all $\kappa$ sufficiently small, 
\begin{align}
\label{estimate-case2-eq3}
&Y_{p+1}({t_0})\leq C_0\kappa^{2+\frac{N(p+m-2)}{p+m+2}} \left[Y_{p+1}(t_0)\right]^{\frac{2(p+m)}{p+m+2}}+C_0\kappa^{-N}.
\end{align}

By further taking $C_0$ to be large enough if necessary, we can assume $Y_{p+1}(0)=X_{p+1}(0)\leq C_0$, we claim that, if  $0<\kappa< 1$ is sufficiently small, 
\begin{equation}
\label{new-claim-eq}
Y_{p+1}(T)\le 2C_0 \kappa^{-N}.
\end{equation}
In fact,
 assume for contradiction that  there is $t_0\in [0,T]$ such that
$$
Y_{p+1}(t_0)=2C_0  \kappa^{-N}.
$$
By \eqref{estimate-case2-eq3},
\begin{align*}
2C_0\kappa^{-N}&\le C_0 \kappa^{2+\frac{N(p+m-2)}{p+m+2}} (2 C_0\kappa^{-N}) ^{\frac{2(p+m)}{p+m+2}}+ C_0\kappa^{-N}\\
&=C_0(2C_0)^{\frac{2(p+m)}{p+m+2}}\kappa^{2-N}+C_0\kappa^{-N}.
\end{align*}
This implies that
$$
1\le (2C_0)^{\frac{2(p+m)}{p+m+2}}\kappa^{2},
$$
which is impossible if $\kappa$ was chosen to be $\frac12 (2C_0)^{-\frac{p+m}{p+m+2}}$. Therefore, the claim \eqref{new-claim-eq} hold true. This completes the proof.
\end{proof}

\subsection{Proof of Lemma \ref{Lp-prior} when $\tau =0$.}\label{Lploc-est0}
In this subsection, we shall establish the local $L^p$ estimate for \eqref{main-perturbed-eq} when $\tau =0$. 

\begin{proof} [Proof of Lemma \ref{Lp-prior} for $\tau =0$:]
Recall that we have the following from previous subsection 
\begin{align*}
&\quad\, \frac{1}{p+1}\frac{d}{dt}\int_{\R^N} u^{p+1}\psi^2  dx\nonumber \\
&\leq \underbrace{-mp\int_{\R^N} u^{p+m-2}|\nabla u|^2\psi^2}_{-I_1} +\underbrace{C_m\kappa\int_{\R^N}u^{p+m-1}|\nabla u|\psi^2}_{\kappa I_2}+ \underbrace{C_m\kappa\eps^{m-1}\int_{\R^N}u^{p}|\nabla u|\psi^2}_{\kappa I_3}\nonumber \\
 &\quad+ \underbrace{|\chi|\kappa\int_{\R^N} u^{p+1}|\nabla v| \psi^2}_{\kappa I_4} +\underbrace{ \int_{\R^N}  (au^{p+1}-b u^{p+2})\psi^2}_{I_5}+\underbrace{ \chi p \int_{\R^N} u^{p} \nabla u\cdot\nabla v\,\psi^2}_{I_6}.
\end{align*}
We also have the following estimates for any $\delta>0$, 
\begin{align}\label{I22-I5}
\kappa I_2
&\le \frac{mp}{4}\int_{\R^N} u^{p+m-2}|\nabla u|^2 \psi^2 + \frac{m}{p}\kappa^2\int_{\R^N}u^{p+m}\psi^2\nonumber\\
\kappa I_3(t) &\leq \frac{mp}{4}\int_{\R^N} u^{p+m-2}|\nabla u|^2\psi^2 +\delta\int_{\R^N}u^{p+2}\psi^2+ C_\delta\kappa^{-N}\nonumber\\
\kappa I_4(t)&\le \delta \int_{\R^N} u^{p+2}\psi^2 +\kappa|\chi|A_p\int_{\R^N}|\nabla v|^{p+2}\psi^2\nonumber\\
 I_5&\le -\int_{\R^N} u^{p+1}\psi^2dx -(b-\delta)\int_{\R^N} u^{p+2}\psi^2 dx+C_\delta\kappa^{-N}.
\end{align}
Again, $A_p := \frac{1}{p+2}\left(\frac{p+2}{p+1}\right)^{-(p+1)}$. Putting these inequalities together gives 
\begin{align}\label{I_2-I5}
&\frac{1}{p+1}\frac{d}{dt}\int_{\R^N} u^{p+1}\psi^2  dx + \int_{\R^N} u^{p+1}\psi^2 +c\int_{\R^N} u^{p+m-2}|\nabla u|^2 \psi^2  \lesssim \k^2 \int_{\R^N}u^{p+m}\psi^2 \nonumber \\
&\qquad\qquad -(b-3\delta)\int_{\R^N} u^{p+2}\psi^2 + \kappa|\chi|A_p\int_{\R^N}|\nabla v|^{p+2}\psi^2 +I_6 + C_\delta \k^{-N}
\end{align}
Similar to the previous section, we  would estimate $I_6$ for different cases.

\medskip

\noindent{\bf Case 1.1: when $1 < m\le 2-\frac{2}{N}$}\label{case1tau0}
We have the following for $\tau =0$
\begin{align*}
   I_6 &=  \chi p \int_{\R^N} u^{p} \nabla u\cdot\nabla v\,\psi^2    = \frac{ \chi p}{p+1} \int_{\R^N} \nabla(u^{p+1}) \cdot\nabla v\,\psi^2 \nonumber\\
   & \le \frac{ -\chi p}{p+1}\int_{\R^N} u^{p+1} \Delta v\,\psi^2 + \frac{ 2\kappa|\chi| p}{p+1}\int_{\R^N} u^{p+1} |\nabla v|\psi^2\nonumber\\
   &= \frac{ -\chi p}{p+1}\int_{\R^N} \lambda u^{p+1} v + \frac{ \mu \chi p}{p+1}\int_{\R^N} u^{p+2}\,\psi^2 + \frac{ 2\kappa|\chi| p}{p+1}\int_{\R^N} u^{p+1} |\nabla v|\psi^2\nonumber\\
    & \le (|\chi|\k+\frac{ \mu |\chi| p}{p+1})\int_{\R^N} u^{p+2}\,\psi^2 + \kappa |\chi| \int_{\R^N}|\nabla v|^{p+2}\psi^2,
\end{align*}
We used the non-negativity of \(u\) and \(v\) to drop the first term, and then applied Young's inequality to obtain the last inequality. Substituting this into \eqref{I_2-I5} and using the fact that \(m<2\) in this case, we obtain
\begin{align}\label{tau0-lp-eq1}
&\frac{1}{p+1}\frac{d}{dt}\int_{\R^N} u^{p+1}\psi^2  dx +
\int_{\R^N} u^{p+1}\psi^2 + c \int_{\R^N}  u^{p+m-2} |\nabla u|^2\psi^2 \nonumber\\
&\le -(b-3\delta -\k^2 -|\chi|\k- \frac{ \mu |\chi| p}{p+1})\int_{\R^N} u^{p+2}\psi^2  + 3 \kappa |\chi| \int_{\R^N} |\nabla v|^{p+2}\psi^2+C_\delta \kappa^{-N} 
\end{align}
By \eqref{nablav-est}, we have that 
\[\int_{\R^N} |\nabla v|^{p+2}\psi^2 \le C\int_{\R^N} u^{p+2}\psi^2.\]
Thus, \eqref{tau0-lp-eq1} becomes 
\begin{align*}
&\frac{1}{p+1}\frac{d}{dt}\int_{\R^N} u^{p+1}\psi^2  dx +
\int_{\R^N} u^{p+1}\psi^2 + c \int_{\R^N}  u^{p+m-2} |\nabla u|^2\psi^2 \nonumber\\
&\le  -(b-3\delta - \k^2 -4 C \kappa |\chi|- \frac{ \mu |\chi| p}{p+1})\int_{\R^N} u^{p+2}\psi^2 +C_\delta \kappa^{-N}. 
\end{align*}
Multiplying through by $e^{(p+1)s}$ and integrating in time together with regularity of our initial data, gives
\begin{align}\label{tau0-eq6}
&\int_{\R^N} u^{p+1}(t)\psi^2 dx + c \iint_{\Omega_t} e^{-(p+1)(t-s)}  u^{p+m-2} |\nabla u|^2\psi^2\nonumber\\
&\qquad \lesssim  -(b-3\delta -\k^2- 4 C \kappa |\chi|- \frac{ \mu |\chi| p}{p+1})\iint_{\Omega_t} e^{-(p+1)(t-s)}u^{p+2}\psi^2 +C_\delta \kappa^{-N}. 
\end{align}
Note that by assumption, $b > \frac{ \mu |\chi| (q-1)}{q} =  \frac{ \mu |\chi| p}{p+1}$ when $q =p+1$. Thus we can choose $\delta$ and $\kappa$ small enough, such that
$$b\ge 3\delta +\k^2 + 4 C \kappa |\chi| + \frac{ \mu |\chi| p}{p+1}.$$
Hence
\begin{equation}\label{ulp'}
    \int_{\R^N} u^{p+1}(t)\psi^2 \lesssim \kappa^{-N}, \qquad \forall t\in [0,T], \qquad 0<p<\frac{b}{(\mu |\chi|- b)_+}.
\end{equation}

\medskip

\noindent{\bf Case 1.2: When $2-\frac{2}{N}<m\le 2$}

\medskip

From \eqref{tau0-eq6} there is a $C>0$ independent of $t, \kappa, \eps$ such that 
\begin{align*}
&\int_{\R^N} u^{p+1}(t)\psi^2 dx + c \iint_{\Omega_t} e^{-(p+1)(t-s)}  u^{p+m-2} |\nabla u|^2\psi^2\lesssim  C\iint_{\Omega_t} e^{-(p+1)(t-s)}u^{p+2}\psi^2 +C_\delta \kappa^{-N}. 
\end{align*}
When $2-\frac{2}{N}< m\le 2$, by Lemma \ref{uLp+2} again, for any $p>1$
\begin{align*}
    \int_{\R^N} u^{p+2}\psi^2 
    &\lesssim \delta\sup_{x_0\in\R^N}\int_{\R^N} |\nabla u^{\frac{p+m}{2}}|^2\psi^{2}(x-x_0) + C_\delta\kappa^{-N}.
\end{align*}
Hence, 
\begin{align*}
&\int_{\R^N} u^{p+1}(t)\psi^2 dx + c \iint_{\Omega_t} e^{-(p+1)(t-s)}  u^{p+m-2} |\nabla u|^2\psi^2 \\
&\lesssim \delta\sup_{x_0\in\R^N}\iint_{\Omega_t} e^{-(p+1)(t-s)}|\nabla u^{\frac{p+m}{2}}|^2\psi^{2}(x-x_0) + C_\delta\kappa^{-N} . 
\end{align*}
Shifting in space and taking supremum over $x_0\in\R^N$ gives 
\begin{align*}
&\sup_{x_0\in\R^N}\int_{\R^N} u^{p+1}(t)\psi^2(x-x_0) dx + c \sup_{x_0\in\R^N}\iint_{\Omega_t} e^{-(p+1)(t-s)}  u^{p+m-2} |\nabla u|^2\psi^2(x-x_0) \\
&\lesssim \delta\sup_{x_0\in\R^N}\iint_{\Omega_t} e^{-(p+1)(t-s)}|\nabla u^{\frac{p+m}{2}}|^2\psi^{2}(x-x_0) + C_\delta\kappa^{-N} . 
\end{align*}
By choosing $\delta$ small enough, we get that $u\in L^\infty(0, T;  L^{p+1}_{loc}(\R^N))$ for any $p >1$. This completes the proof of this case.

\medskip

\noindent{\bf Case 2: When $m>2$}

\medskip

Just as in the case $\tau =1$, we estimate $I_6$ as 
\begin{align*}
   I_6 &\le \frac{mp}{4}\int_{\R^N} u^{p+m-2}|\nabla u|^2 \psi^2 + \delta \int_{\R^N} u^{p+2} \psi^2 + C_\delta\int_{\R^N} |\nabla v|^{\frac{2(p+2)}{m}} \psi^2 
\end{align*}
Combining this with \eqref{I22-I5} gives 
\begin{align*}
&\frac{1}{p+1}\frac{d}{dt}\int_{\R^N} u^{p+1}\psi^2  dx + \int_{\R^N} u^{p+1}\psi^2 +c\int_{\R^N} u^{p+m-2}|\nabla u|^2 \psi^2  \lesssim \k^2 \int_{\R^N}u^{p+m}\psi^2 \nonumber \\
&\qquad\qquad +C\int_{\R^N} u^{p+2}\psi^2 + \kappa|\chi|A_p\int_{\R^N}|\nabla v|^{p+2}\psi^2 +C_\delta\int_{\R^N} |\nabla v|^{\frac{2(p+2)}{m}} \psi^2+ C_\delta \k^{-N}\nonumber\\
&\qquad\qquad\lesssim \k^2 \int_{\R^N}u^{p+m}\psi^2 +\k\int_{\R^N}|\nabla v|^{p+2}\psi^2 +C_\delta\k^{-N}
\end{align*}
We used Young's inequality and the fact that $m>2$, to get the last inequality. Again by \eqref{nablav-est}, we have that 
\[\int_{\R^N} |\nabla v|^{p+2}\psi^2 \le C\int_{\R^N} u^{p+2}\psi^2\le \delta_1\int_{\R^N}u^{p+m}\psi^2 + C_{\delta_1}\k^{-N}.\]
Thus, for $\delta_1< \k$, we have
\begin{align*}
\frac{1}{p+1}\frac{d}{dt}\int_{\R^N} u^{p+1}\psi^2  dx + \int_{\R^N} u^{p+1}\psi^2 +c\int_{\R^N} u^{p+m-2}|\nabla u|^2 \psi^2  
\lesssim \k^2 \int_{\R^N}u^{p+m}\psi^2  +C_\delta\k^{-N}
\end{align*}
By \eqref{lup+m}
\begin{equation*}
   \kappa^2\int_{\R^N}u^{p+m}\psi^2
\lesssim  \frac{c}{2} \int_{\R^N} u^{p+m-2}|\nabla u|^2\psi^2dx+ C_{\delta}\kappa^{2+{\theta^*}}X_{p+1}(s)^{2/q^*}+C\kappa^{2-N}. 
\end{equation*}
Hence, these two inequalities implies 
\begin{align*}
\int_{\R^N}u^{p+1}(t,x)\psi^2(x)dx+c/2\iint_{\Omega_{t}} e^{-\lambda(p+1)(t-s)}u^{p+m-2}|\nabla u|^2\psi^2dxds \lesssim    \kappa^{-N} +  C_{\delta}\kappa^{2+{\theta^*}}Y_{p+1}(t)^{2/q^*}.
\end{align*}
Shifting in space and taking supremum over $x_0$ in $\R^N$ gives
\begin{align*}
&\sup_{x_0\in\R^N}\int_{\R^N}u^{p+1}(t,x)\psi^2(x-x_0)dx\lesssim  \kappa^{-N}+  C_{\delta}\kappa^{2+{\theta^*}}Y_{p+1}(t)^{2/q^*}
\end{align*}
Taking supremum in $t\in [0,t_0]$ for any $t_0\in [0,T]$,  there exists $C_0>0$ independent of $t_0$ such that for all $\kappa$ sufficiently small, 
\begin{align*}
&Y_{p+1}({t_0})
\leq C_0\kappa^{2+\frac{N(p+m-2)}{p+m+2}} \left[Y_{p+1}(t_0)\right]^{\frac{2(p+m)}{p+m+2}}+C_0\kappa^{-N}.
\end{align*}
By same argument as in the case of $\tau =1$, if  $0<\kappa< 1$ is sufficiently small, 
\begin{equation*}
Y_{p+1}(T)\le 2C_0 \kappa^{-N}.
\end{equation*}
This completes the proof.
\end{proof}

\subsection{Proof of Proposition 3.1 (Sketch)}
\begin{proof}
The proof is based on a standard fixed-point argument, together with the theory of quasilinear parabolic equations developed in
\cite{amann1993nonhomogeneous,ladyzhenskaya1968linear}, applied in a suitable fixed-point framework. For the case \(\tau=1\), the proof follows by an argument similar to that of Proposition 1.1 in \cite{hassan2025global} for the consumption case \eqref{special-eq3}. See also \cite{sugiyama2006global,sugiyama2007time,wang2014quasilinear} for further details on fixed-point methods and the existence of classical solutions to perturbed problems of the form \eqref{main-perturbed-eq}.

We first show that the condition \(b>b_m\) implies that \(u\) is locally in \(L^p\) for every \(p>1\), using Lemma 3.1 as follows.

\medskip

\noindent
\textbf{The case \(\tau=1\).}
Assume first that \(1<m\le \frac{2N}{N+2}\). Since $b>b_m = \left(\inf_{{p} >\max\{1,\frac{N(2-m)}{2m}\}}\frac{p-1}{p}\left(C_{p+1,N}\right)^{\frac{1}{p+1}}\right)|\chi|\mu $
we may choose \(p'\) such that
\(
p'+1>\frac{N(2-m)}{2m}
\)
and
\(
b>
\mu|\chi|
\left(
\frac{p'}{p'+1}
\left(C_{p'+2,N}\right)^{\frac{1}{p'+2}}
\right).
\)
Hence, by Lemma 3.1,
\(
u\in L^\infty\big([0,T];L^{p'+1}_{\mathrm{loc}}(\mathbb R^N)\big).
\)
We claim that for every \(p>p'\), \(u\) is locally in \(L^{p+1}\). Indeed, from \eqref{eq:u}, there exists a constant \(C>0\) such that
\begin{align*}
&\int_{\mathbb R^N} u^{p+1}(t,x)\psi^2(x)\,dx
+\frac{mp}{2}
\iint_{\Omega_t}
e^{-\lambda(p+1)(t-s)}
u^{p+m-2}|\nabla u|^2\psi^2\,dx\,ds
\\
&\qquad\lesssim
C\iint_{\Omega_t}
e^{-\lambda(p+1)(t-s)}
u^{p+2}\psi^2\,dx\,ds
+
C_\delta \kappa^{-N}.
\end{align*}
Applying Lemma \ref{uLp+2} to the first term on the right-hand side, we obtain, for any \(\delta>0\),
\begin{align*}
&\int_{\mathbb R^N} u^{p+1}(t)\psi^2\,dx
+\frac{mp}{2}
\iint_{\Omega_t}
e^{-\lambda(p+1)(t-s)}
u^{p+m-2}|\nabla u|^2\psi^2\,dx\,ds
\\
&\qquad\lesssim
\delta
\sup_{x_0\in\mathbb R^N}
\iint_{\Omega_t}
e^{-\lambda(p+1)(t-s)}
\left|\nabla u^{\frac{p+m}{2}}\right|^2
\psi^2(x-x_0)\,dx\,ds
+
C_\delta\kappa^{-N}.
\end{align*}
Shifting in space, taking the supremum over \(x_0\in\mathbb R^N\), and choosing \(\delta>0\) sufficiently small, we obtain
\(
u\in L^\infty\big(0,T;L^{p+1}_{\mathrm{loc}}(\mathbb R^N)\big)
\), for any $p >1$ when $b>b_m$ and $\tau =1$.

\medskip

\noindent
\textbf{The case \(\tau=0\).}
Assume that \(1<m\le 2-\frac{2}{N}\) since $b>b_m =\frac{(2-m)N -2}{(2-m)N}$ it follows that $$0\le \frac{N(2-m)}{2}- 1 < \frac{b}{(\mu |\chi|- b)_+}$$ Now choose a $p'\in \left(\frac{N(2-m)}{2}- 1 , \frac{b}{(\mu |\chi|- b)_+}\right)$, then $b>b_{m, p'+1}$ and $u\in L^\infty([0,T]; L^{p'+1}_{loc}(\R^N))$ by Lemma 3.1. Moreover, from \eqref{tau0-eq6}, for every \(p>p'\) there exists a constant \(C>0\), independent of \(t\), \(\kappa\), and \(\varepsilon\), such that
\begin{align*}
&\int_{\mathbb R^N} u^{p+1}(t)\psi^2\,dx
+
c\iint_{\Omega_t}
e^{-(p+1)(t-s)}
u^{p+m-2}|\nabla u|^2\psi^2\,dx\,ds
\\
&\qquad\lesssim
C\iint_{\Omega_t}
e^{-(p+1)(t-s)}
u^{p+2}\psi^2\,dx\,ds
+
C_\delta\kappa^{-N}.
\end{align*}
Applying Lemma \ref{uLp+2} again to the first term on the right-hand side, we obtain, for every \(p>p'\) and every \(\delta>0\),
\begin{align*}
&\int_{\mathbb R^N} u^{p+1}(t)\psi^2\,dx
+
c\iint_{\Omega_t}
e^{-(p+1)(t-s)}
u^{p+m-2}|\nabla u|^2\psi^2\,dx\,ds
\\
&\qquad\lesssim
\delta
\sup_{x_0\in\mathbb R^N}
\iint_{\Omega_t}
e^{-(p+1)(t-s)}
\left|\nabla u^{\frac{p+m}{2}}\right|^2
\psi^2(x-x_0)\,dx\,ds
+
C_\delta\kappa^{-N}.
\end{align*}
Shifting in space and taking the supremum over \(x_0\in\mathbb R^N\), we get
\begin{align*}
&\sup_{x_0\in\mathbb R^N}
\int_{\mathbb R^N} u^{p+1}(t)\psi^2(x-x_0)\,dx
+
c\sup_{x_0\in\mathbb R^N}
\iint_{\Omega_t}
e^{-(p+1)(t-s)}
u^{p+m-2}|\nabla u|^2\psi^2(x-x_0)\,dx\,ds
\\
&\qquad\lesssim
\delta
\sup_{x_0\in\mathbb R^N}
\iint_{\Omega_t}
e^{-(p+1)(t-s)}
\left|\nabla u^{\frac{p+m}{2}}\right|^2
\psi^2(x-x_0)\,dx\,ds
+
C_\delta\kappa^{-N}.
\end{align*}
Choosing \(\delta>0\) sufficiently small, we obtain
\[
u\in L^\infty\big(0,T;L^{p+1}_{\mathrm{loc}}(\mathbb R^N)\big)
\]
for every \(p>p'\). As above, this yields local \(L^p\)-bounds for all \(p>1\).

The a priori estimates in Lemmas 2.5 and Lemma 2.6, together with the theory of quasilinear parabolic equations, can now be used to carry out a standard fixed-point argument. We refer to the proof of Proposition 1.1 in \cite{hassan2025global} for the details.
\end{proof}

\section{Global Existence of a Weak Solution}
In this section, we prove Theorems 1.1 and 1.2 using Proposition 3.1. First, we prove the existence of a global weak solution for any \(m>1, \ b>0\) using Simon--Dubinski\u{\i} compactness theorem in Subsection 4.1. Then, in Subsection 4.2, we use the global boundedness of solutions to the perturbed problem to establish the existence of a globally bounded weak solution for \(m\ge 2\), and also for \(1<m<2\) when \(b\) is sufficiently large.

\subsection{Existence of a Weak Solution/Proof of Theorem 1.1}
Consider the following perturbed version of \eqref{main-eq}:
\begin{equation}
\label{main-perturbed-eq2}
\begin{cases}
u_t = m\nabla\cdot \big((\eps+u)^{m-1}\nabla u\big) - \chi \nabla \cdot (u \nabla v) + u(a - b u)- \eps u^{\sigma},\quad & x \in \mathbb{R}^N, \,\, t>0, \\
\tau v_t = \Delta v - \lambda v + \mu u, \quad & x \in \mathbb{R}^N,\,\,  t>0,\\
u(0,x) = u_0(x),\,\, \tau v(0, x) = \tau v_0(x), \quad & x\in \R^N,
\end{cases}
\end{equation}
where \(\eps\in (0,1)\) and \(\sigma>2\). Here, $f(u) = u(a - b u)- \eps u^{\sigma} \le u(a_1-b_1u)$, for any $b_1>0$ and for some $a_1= a_1(\eps, a, b, b_1)>0$ . Assume that \(u_0\ge 0\), and when \(\tau=1\), also \(v_0\ge 0\). Suppose moreover that
\(
u_0\in C_{\rm unif}^{1+\alpha}(\R^N)
\)
for some \(\alpha\in(0,1)\), and, in the case \(\tau=1\),
\(
v_0\in C_{\rm unif}^{2+\alpha}(\R^N).
\)
Then, by Proposition (1.1) for each \(\eps\in(0,1)\), by choosing $b_1>b_{m,p}$ big enough, problem \eqref{main-perturbed-eq2} admits a unique nonnegative global classical solution \((u_\eps,v_\eps)\) of \eqref{main-perturbed-eq2} (Note that this by lemma 2.5 and Lemma 2.6 will satisfy all the regularity in the lemmas but depend on $\eps$ since $a_1$ depend on $\eps$). Now we would show that we can get some $\eps$ independent estimates of the solution.

 \medskip
First assume $1<m<2$. {\bf In the case \(\tau = 1\)}, take 
\begin{equation}\label{p0-eq2}
   p_0\in\Big[m, \min\{2,\frac{b}{\left(b-2C_{p_0+2, N}^{\frac{1}{p_0+2}}\mu|\chi|\right)_+}\}\Big)\qquad \text{and} \qquad \sigma=1+p_0,
\end{equation}
then
\[
\frac{b}{2}>\frac{p_0}{p_0+1}C_{p_0+2,N }^{\frac{1}{p_0 +2}}\mu |\chi|.
\]
Since
\[
\min_{r>0} \left(r+r^{-(p_0+1)} \left(\frac{\mu|\chi|p_0}{p_0+1}\right)^{p_0+2}A_{p_0}C_{p_0+2,N }\right)
=
\frac{p_0}{p_0+1}\left(C_{p_0+2,N}\right)^{\frac{1}{p_0+2}}|\chi|\mu,
\]
then for some \(r>0\), and \(\delta\) small enough,
\[
\frac{b}{2}-9\delta - r-r^{-(p+1)} \left(\frac{\mu|\chi|p_0}{p_0+1}\right)^{p_0+2}A_{p_0}C_{p_0+2,N }\ge 0.
\]
Hence, by the same argument as in the case \(m<\frac{2N}{N+2}\) in Section \ref{Lploc-est1}, we have
\begin{align}\label{eq:u'}
&\int_{\R^N} u_\eps^{p_0+1}(t)\psi^2 \, dx + c \iint_{\Omega_t} e^{-\lambda(p_0+1)(t-s)}  (\eps +u_\eps)^{p_0+m-2} |\nabla u_\eps|^2\psi^2 + \frac{b}{2} \iint_{\Omega_t} e^{-\lambda(p_0+1)(t-s)}u_\eps^{p_0+2}\psi^2\nonumber\\
& \lesssim  -\left(\frac{b}{2}-9\delta - r-r^{-(p+1)} \left(\frac{\mu|\chi|p_0}{p_0+1}\right)^{p_0+2}A_{p_0}C_{p_0+2,N }\right)\iint_{\Omega_t} e^{-\lambda(p_0+1)(t-s)}u_\eps^{p_0+2}\psi^2  \nonumber\\
&\quad- \eps\iint_{\Omega_t}e^{-\lambda(p_0+1)(t-s)}u_\eps^{p_0+\sigma}\psi^2+C_\delta \kappa^{-N} \nonumber\\
& \le C_\delta \kappa^{-N}.
\end{align}

\medskip

\noindent {\bf In the case \(\tau = 0\),} take \(\sigma= p_0 +1\), with  
\begin{equation}\label{p0-eq1}
    p_0 \in [m, \min\{2,\frac{b}{\left(b-2\mu|\chi|\right)_+}\})\subset(1,2).
\end{equation}
This gives that
\[
\frac{b}{2}>\frac{\mu |\chi|p_0}{p_0+1},
\]
then for \(\delta\) and \(\kappa\) small enough,
\[
\frac{b}{2}-4\delta - 3 C \kappa |\chi|- \frac{ \mu |\chi| p_0}{p_0+1}\ge 0.
\]
By the same argument as in the case \(m<2-\frac{2}{N}\) in Section 4.2, we obtain the following local estimate for the \(L^{p_0 +1}\)-norm of \(u_\eps\):
\begin{align}\label{tau0-pertub}
&\int_{\R^N} u_\eps^{p_0+1}(t)\psi^2 \, dx + c \iint_{\Omega_t} e^{-\lambda(p_0+1)(t-s)}  (\eps +u_\eps)^{p_0+m-2} |\nabla u_\eps|^2\psi^2 + \frac{b}{2} \iint_{\Omega_t} e^{-\lambda(p_0+1)(t-s)}u_\eps^{p_0+2}\psi^2\nonumber\\
&\qquad \lesssim  -\left(\frac{b}{2}-4\delta - 3 C \kappa |\chi|- \frac{ \mu |\chi| p_0}{p_0+1}\right)\iint_{\Omega_t} e^{-\lambda(p_0+1)(t-s)}u_\eps^{p_0+2}\psi^2  \nonumber\\
&\qquad\quad - \eps\iint_{\Omega_t}e^{-\lambda(p_0+1)(t-s)}u_\eps^{p_0+\sigma}\psi^2+C_\delta \kappa^{-N} \nonumber\\
&\qquad \le C_\delta \kappa^{-N}.
\end{align}

Hence, in both cases,
\(
u_\eps\in L^\infty(0, T; L^{p_0+1}_{loc}(\R^N))\cap L^{p_0+2}\big(0,T; L^{p_0+2}_{loc}(\R^N)\big)\),
and \(
\nabla (\eps +u_\eps)^{\frac{p_0 + m}{2}}\in L^2\big((0,T); L^2_{loc} (\R^N)\big).
\)
In particular, for \(p_0 = m\), we have
\(
\nabla (\eps+ u_\eps)^{m}\in L^2\big((0,T); L^2_{loc} (\R^N)\big).
\)
Next we prove the following important lemma.

\begin{lem}\label{cross-lm}
    Let $(u_\eps, v_\eps)$ be a classical solution to \eqref{main-perturbed-eq2}, $p_0$ be as in \eqref{p0-eq1} and \eqref{p0-eq2}, then there exists a constant $C >0$ such that 
\begin{equation}\label{eq:5.7}
\sup_{0\le t\le T}\|(\nabla v_\eps(t,\cdot))\psi\|_{L^{p_0 +1}} + \sup_{0\le t\le T}\|v_\eps(t,\cdot)\psi\|_{L^{p_0 +1}}\le C.
\end{equation}
Also, there exists a $C(T)$ but independent of $\eps$, and a $p\in(1, 2)$, such that
    \begin{align}\label{cross}
        \int_0^T\int_{\R^N} |u_\eps\nabla v_\eps|^p \psi + \int_0^T\int_{\R^N} |f(u_\eps)|^p\psi\le C(T), 
    \end{align}
    where $f(u_\eps) = au_\eps - bu_\eps^2 -\eps u_\eps^\sigma$.
\end{lem}

\begin{proof}
 Let $p_0$ be as in \eqref{p0-eq1} and \eqref{p0-eq2} for $\tau = 0$ and $\tau = 1$ respectively. To prove \eqref{eq:5.7}, consider the proof of property 1 in Lemma \ref{v-bound-lm}. Using semigroup arguments in \eqref{main-eq3-psi} to \eqref{eq:3.11} for $\tau =0$ and \eqref{v-phi-eq} to \eqref{A-eq1} for $\tau =1$ with $p=p'=p_0+1$, together with the fact that $u_\eps\in  L^\infty(0, T; L^{p_0+1}_{loc}(\R^N))$ gives that for $\kappa$ sufficiently small,
\begin{equation}
\label{A-eq1'}
\sup_{0\le t\le T}\|(\nabla v_\eps(t,\cdot))\psi\|_{L^{p_0+1}} + \sup_{0\le t\le T}\|v_\eps(t,\cdot)\psi\|_{L^{p_0+1}}\le C_\kappa.
\end{equation}

Now for \eqref{cross}, by Young's inequality,
\begin{align}
    \int_0^T\int_{\R^N} |u_\eps\nabla v_\eps|^{\frac{N(p_0+1)}{2N-p_0}}\psi\le \int_0^T\int_{\R^N} u_\eps^{p_0 +1}\psi + \int_0^T\int_{\R^N}|\nabla v_\eps |^{\frac{N(p_0 +1)}{N-p_0}}\psi
\end{align}
Using the semigroup estimate \eqref{Lp Estimates-3}, we can get a $C(T)>0$ such that 
\begin{align}
   \int_0^T\int_{\R^N}|\nabla v_\eps |^{\frac{N(p_0 +1)}{N-p_0}}\psi\le  C(T) + \int_0^T\int_{\R^N} u_\eps^{p_0 +1}\psi 
\end{align}
Combining these inequalities and the fact that $u_\eps\in  L^\infty(0, T; L^{p_0+1}_{loc}(\R^N))$ gives that we can get a $1<p<2$ such that
\begin{align*}
    \int_0^T\int_{\R^N} |u_\eps\nabla v_\eps|^p\psi\le C(T). 
\end{align*}
Also, it is clear by Young's inequality, that we can find a $C>0$ such that $|f(u_\eps)|\le C+ Cu^\sigma$. Hence
$$\int_0^T\int_{\R^N}|f(u_\eps)|^{\frac{p_0+2}{\sigma}} \psi\lesssim CT\k^{-N} + C\int_0^T\int_{\R^N} u_\eps^{p_0 +2}\psi \lesssim C(T).$$
The last inequality follows since $u\in L^{p_0+2}((0,T); L^{p_0+2}_{loc}(\R^N))$. Hence for $p =\frac{p_0+2}{\sigma}\in(1,2)$, there exists a constant $C(T)>0$ such that 
$$\int_0^T\int_{\R^N}|f(u_\eps)|^p\psi\lesssim C(T).$$
\end{proof}

Note that for $m\ge 2$, Lemma 2.5 and Lemma 2.6, together with Lemma 3.1, gives a better regularity than the ones in the above lemma. Now we give the proof of Theorem 1.1

 \begin{proof}[Proof of Theorem 1.1]
Let the assumption of Theorem \ref{main-thm1} holds. Let $\tau \in \{0, 1\}$, $\eta\in C_c^\infty(\R^N)$ be non-negative with unit total mass. For any $\eps\in (0,1)$, define $\eta_\eps(x)=\eps^{-N}\eta(\eps^{-1}x)$, and let
$
u_{0,\eps}=u_0*\eta_\eps$ and 
$\tau v_{0,\eps}=\tau v_0*\eta_\eps$. Consequently, for any $p\geq 1$, we have
\[
\|\tau (v_{0,\eps}-v_0)\|_{L^{p}_{\rm loc}},\, \|\tau(\nabla v_{0,\eps}-\nabla v_0)\|_{L^{p}_{\rm loc}},\, \|u_{0,\eps}-u_0\|_{L^{p}_{\rm loc}}\to 0
\quad \text{ as }\eps\to 0.
\]
Moreover, $u_{0,\eps}\in L^\infty(\R^N)$, $\tau v_{0,\eps}\in W^{1,\infty}$, and
$$
\|u_{0,\eps}\|_\infty\le \|u_0\|_\infty,\quad \|\tau v_{0,\eps}\|_{W^{1,\infty}}\le \|\tau v_0\|_{W^{1,\infty}}.
$$ 
Let $(u_\eps, v_\eps)$ be a classical solution to \eqref{main-perturbed-eq2} with initial $u_{0\eps}, \tau v_{0, \eps}$ in place of $u_0$ and $\tau v_0$. 

First, let $1<m<2$ and $p_0$ be as in \eqref{p0-eq1} and \eqref{p0-eq2}. From $u_\eps \in L^\infty(0, T; L^{p_0+1}_{loc}(\R^N))\cap L^{p_0+2}(0,T); L^{p_0+2}_{loc}(\R^N)) $,  and $\nabla (\eps +u_\eps)^{m}\in L^2((0,T); L^2_{loc} (\R^N))$
up to a subsequence, we have for some $p\in(1, 2)$ there exists a $u\in L^p(0,T; L^p_{loc} (\R^N))$ with $\nabla u^m \in L^p(0,T; L^p_{loc} (\R^N))$, such that 
\begin{align}\label{um-bound}
u_\eps \rightharpoonup u, &\qquad in \quad L^p(0,T; L^p_{loc} (\R^N)),\nonumber\\
\nabla (\eps +u_{\eps})^m \rightharpoonup \nabla u^m &\qquad in\quad L^p(0,T; L^p_{loc} (\R^N)).
\end{align}
Also from equation \eqref{cross} in Lemma~\ref{cross-lm}, up to a subsequence, we can find a $z\in L^p(0,T; L^p_{loc} (\R^N))$ such that
\begin{align}\label{uepsnablav}
u_{\eps }\nabla v_{\eps } \rightharpoonup z 
&\qquad \text{in } L^p(0,T; L^p_{\mathrm{loc}} (\R^N)).
\end{align}

Fix \(R>0\) and write \(Q_{R,T}:=(0,T)\times B_R\).  Now choose \(p_0=m\), then $u_\varepsilon$ is bounded in
\( 
L^\infty\big(0,T;L^{m+1}(B_R)\big)
\cap L^{m+2}(Q_{R,T}),
\)
and
\(
\nabla (\varepsilon+u_\varepsilon)^m
\ \text{is bounded in }\ 
L^2(Q_{R,T}).
\)
Moreover, since \(m<2\), we have \(2m < m+2\), and therefore
\[
u_\varepsilon^m ,\qquad |\nabla u_\varepsilon^m| \ \text{is bounded in } \ L^2(Q_{R,T}).
\]
This implies
\[
u_\varepsilon^m \ \text{is bounded in } \ L^2\big(0,T;W^{1,2}(B_R)\big).
\]
Now define the seminorm
\[
\mathcal M(B_R):=\{w\ge 0:\ w^m\in W^{1,2}(B_R)\},
\qquad
\|w\|_{\mathcal M(B_R)}:=\|w^m\|_{W^{1,2}(B_R)}^{1/m}.
\]
Then
\(
\{u_\varepsilon\}_\varepsilon
\ \text{is bounded in }\ 
L^{2m}\big(0,T;\mathcal M(B_R)\big),
\)
because
\[
\int_0^T \|u_\varepsilon(t)\|_{\mathcal M(B_R)}^{2m}\,dt
=
\int_0^T \|u_\varepsilon^m(t)\|_{W^{1,2}(B_R)}^2\,dt
\le C(R,T).
\]
Next, we claim that
\[
\mathcal M(B_R)\hookrightarrow\hookrightarrow L^{2m}(B_R).
\]
Indeed, if \(\{w_n\}\subset \mathcal M(B_R)\) is bounded, then
\(
z_n:=w_n^m
\)
is bounded in \(W^{1,2}(B_R)\). By the Rellich theorem, up to a subsequence, $z_n\to z$ strongly in $L^2(B_R)$.
Let \(w:=z^{1/m}\). Since for \(a,b\ge 0\),
\(
|a-b|^{2m}\le |a^m-b^m|^2
\), 
it follows that
\[
\|w_n-w\|_{L^{2m}(B_R)}^{2m}
\le
\|z_n-z\|_{L^2(B_R)}^2\to 0.
\]
Hence the embedding is compact.

Let \(p\in(1,2)\) be the exponent from Lemma \ref{cross-lm}, and let \(q\) be its H\"older conjugate. From \eqref{cross}, we have that 
\(
u_\varepsilon\nabla v_\varepsilon\) and 
\(
g(u_\varepsilon)
\) are bounded in \( 
L^p(Q_{R,T}).\) This together with the fact that \(\nabla(\varepsilon+u_\varepsilon)^m\) is bounded in \(L^2(Q_{R,T})\) gives that for
\(\zeta\in W_0^{1,q}(B_R)\), we obtain
\begin{align*}
\left|\int_{B_R} (u_\varepsilon)_t \zeta \,dx\right|
&\le
\left(
\|\nabla(\varepsilon+u_\varepsilon)^m\|_{L^p(B_R)}
+
\|u_\varepsilon\nabla v_\varepsilon\|_{L^p(B_R)}
+
\|g(u_\varepsilon)\|_{L^p(B_R)}
\right)
\|\zeta\|_{W_0^{1,q}(B_R)} .
\end{align*}
It follows that
\(
\partial_t u_\varepsilon
\ \text{is bounded in }\ 
L^p\big(0,T;(W_0^{1,q}(B_R))^*\big).
\)
On the other hand,
\[
L^{2m}(B_R)\hookrightarrow (W_0^{1,q}(B_R))^*
\]
continuously. Indeed, since \(q>2\), one has \((2m)'<2<q\), and therefore
\[
\|\phi\|_{L^{(2m)'}(B_R)}
\le C \|\phi\|_{W_0^{1,q}(B_R)}
\qquad \text{for all } \phi\in W_0^{1,q}(B_R).
\]
Therefore all assumptions of the Simon--Dubinski\u{\i} compactness theorem (see \cite{chen2012two, chen2014note}) are satisfied with
\[
\mathcal M(B_R)\hookrightarrow\hookrightarrow L^{2m}(B_R)
\hookrightarrow (W_0^{1,q}(B_R))^*,
\]
and consequently \(\{u_\varepsilon\}_\varepsilon\) is relatively compact in
\(
L^{2m}\big(0,T;L^{2m}(B_R)\big).
\)
 Hence, along a subsequence,
\[
u_\varepsilon \to u
\qquad\text{strongly in }L^{2m}\big(0,T;L^{2m}(B_R)\big).
\]
Finally, by a diagonal argument over \(R=1,2,\dots\), we obtain
\(
u_\varepsilon \to u
\) strongly in $L^{2m}_{\mathrm{loc}}\big((0,T)\times\R^N\big)$.
Now from \eqref{eq:5.7} with \(p_0=m\), we have
\begin{align*}
&v_\varepsilon \rightharpoonup v
\qquad\text{weakly in }L^{m+1}_{\mathrm{loc}}\big((0,T)\times\R^N\big)\\
  &\nabla v_\varepsilon \rightharpoonup \nabla v
\qquad\text{weakly in }L^{m+1}_{\mathrm{loc}}\big((0,T)\times\R^N\big).  
\end{align*}
Since
\(
(m+1)'=\frac{m+1}{m}<2m,
\)
the above strong convergence implies
\[
u_\varepsilon \to u
\qquad\text{strongly in }L^{(m+1)'}_{\mathrm{loc}}\big((0,T)\times\R^N\big).
\]
Therefore, for every \(\Phi\in C_c^\infty((0,T)\times\R^N;\R^N)\),
\begin{align*}
\int_{\R^N} (u_\varepsilon\nabla v_\varepsilon-u\nabla v)\cdot \Phi
&=
\int_{\R^N} \nabla v_\varepsilon\cdot (u_\varepsilon-u)\Phi
+
\int_{\R^N} (\nabla v_\varepsilon-\nabla v)\cdot u\Phi
\to 0.
\end{align*}
Hence,
\(
u_\varepsilon\nabla v_\varepsilon \to u\nabla v\) in 
\(\mathcal D'((0,T)\times\R^N).
\) 
In particular, if
\(
u_\varepsilon\nabla v_\varepsilon \rightharpoonup z\)
weakly in \(L^p_{\mathrm{loc}}((0,T)\times\R^N)
\), 
then necessarily
\[
z=u\nabla v.
\]
Hence, after passing limit as $\eps \to 0$ along a subsequence, we get that \((u,v)\) is a weak solution of \eqref{main-eq} on \([0,T]\times \R^N\) satisfying,
\[
u\in L^\infty\big(0,T;L^{m+1}_{\mathrm{loc}}(\R^N)\big)
\cap L^{m+2}_{\mathrm{loc}}\big((0,T)\times \R^N\big),
\quad
\nabla u^m\in L^2_{\mathrm{loc}}\big((0,T)\times \R^N\big), \quad v,\ \nabla v \in L^\infty\big(0,T;L^{m+1}_{\mathrm{loc}}(\R^N)\big),
\]
and
\[
u\nabla v \in L^p_{\mathrm{loc}}\big((0,T)\times \R^N\big)
\]
for some \(p\in(1,2)\).
When $m\ge 2$, the approximate solutions enjoy better regularity: 
$u_\eps$ is uniformly bounded and uniformly H\"older continuous with bounds independent of $\eps$. 
Consequently, up to a subsequence, $u_\eps$ converges strongly, locally uniformly in 
$(0,\infty)\times \R^N$. Moreover, $v_\eps$ and $\nabla v_\eps$ are uniformly bounded. 
These estimates yield the existence of a globally bounded weak solution in this case. 
The details are given below in the proof of Theorem \ref{main-thm}.
\end{proof}

\subsection{Existence of a Bounded Weak Solution/Proof of Theorem 1.2}
In this subsection we prove the existence of a globally bounded non-negative weak solution. 
\begin{proof}[Proof of Theorem 1.2]
Let the assumption of Theorem 1.2 holds. Let $(u_\eps,v_\eps)$ be classical solution of \eqref{main-perturbed-eq} with $f(u) = au - bu^2$ and initials $u_{0,\eps}$ and $\tau v_{0,\eps}$ as in Theorem 1.1 (this  exist uniquely from Proposition \ref{main-perturbed-thm}). For any $T>0$, by Lemma \ref{v-bound-lm} and Lemma \ref{apriori-prop}, there exists a constant  $C>0$, independent of $\eps$ and $T>0$ such that
\beq\lb{678}
\|u_{\eps}\|_{L^\infty(\Omega_{T})}  \le C,\qquad \|v_\eps\|_{W^{1,\infty}(\Omega_{T})}  \le C.
\eeq
Therefore, after passing $\eps\to 0$ along a subsequence, 
we can find $v\in L^\infty(0,T;W^{1,\infty}(\R^N))$ and  $u\in L^\infty(\Omega_{T})$ such that
\beq\lb{convergence}
\begin{aligned}
&v_\eps\rightharpoonup v,\quad
\nabla v_\eps\rightharpoonup \nabla v,
\quad u_\eps \rightharpoonup u &&\quad\text{ in }L^p_{\rm loc}(\Omega_{T})\text{ for all }p\geq 1. 
\end{aligned}
\eeq
Moreover, Lemma \ref{holdercty-lem},   gives that $u_\eps$ is uniformly H\"{o}lder continuous in $[T',\infty]\times\R^N$ with fixed $T'>0$, and the H\"{o}lder norm is independent of $\eps\in (0,1)$. This shows that along a subsequence,
\begin{equation*}
    u_\eps\to u \quad \text{point wise locally uniformly in } (0,\infty)\times\R^N.
\end{equation*}
Thus,  $u\in C((0,T); L^2_{loc}(\R^N))$. This convergence together with \eqref{convergence} implies that 
\[
u_\eps \nabla v_\eps \rightharpoonup u\nabla v\quad\text{ in }L^p_{\rm loc}(\Omega_{T})\text{ for all }p\geq 1.
\]
Also, by property (3) in Lemma \ref{v-bound-lm}, with $p=m$, there exists a constant $M>0$, independent of $\eps$ and $T$ such that 
\[
\sup_{x_0\in\R^N}\iint_{[0,T]\times B_1(x_0)}|\nabla (u_\eps+\eps)^m(t,x)|^2dxdt\leq M.
\]
Therefore, after passing $\eps\to 0$ along a subsequence, 
we can find $v\in L^\infty(0,T;W^{1,\infty}(\R^N))$ and  $u\in L^\infty(\Omega_{T})$ such that
\beq
\begin{aligned}
&u_\eps\rightharpoonup u,\quad v_\eps\rightharpoonup v,\quad
\nabla v_\eps\rightharpoonup \nabla v,
\quad u_\eps \nabla v_\eps \rightharpoonup u\nabla v &&\quad\text{ in }L^p_{\rm loc}(\Omega_{T})\text{ for all }p\geq 1,
\\
&\nabla(u_\eps +\eps)^{m}\rightharpoonup \nabla u^{m} &&\quad\text{ in }L^2_{\rm loc}(\Omega_{T}).    
\end{aligned}
\eeq
Hence, $(u,v)$ is a global weak solution of \eqref{main-eq}, and they stay uniformly bounded for all time. Finally, since all the estimates in Lemma \ref{v-bound-lm} holds independent of $\eps$, letting $\eps$ goes to zero, we obtain that all the regularity properties of the solution $(u, v)$ in Theorem \ref{main-thm} holds. 
\end{proof}

\section{Regularity and Uniqueness of Bounded Weak Solutions}

In this section, we prove Theorem \ref{uniqueness}, which establishes the H\"older continuity and uniqueness of weak solutions that are H\"older continuous up to the initial time.
 
\subsection{H\"older Continuity/Proof of Theorem 1.3(1)}
First we established that any globally bounded weak solution of \eqref{main-eq} is H\"older continuous, and the H\"older continuity hold up to initial time for H\"older continuous initials.

\begin{proof}[Proof of Theorem 1.3(1)]
    Let $(u, v)$ be a globally bounded weak solution of \eqref{main-eq} satisfying the assumptions in Theorem 1.3(1). Such a weak solution exists by Theorem 1.2. Therefore, interior H\"older continuity of $u$ follows from  \cite[Theorems 1.3]{black2026refining}. Moreover, if \(u_0\in C^\alpha(\R^N)\), then
\(
u\in C^\alpha([0,T]\times\R^N)
\)
by \cite[Theorems 1.9]{black2026refining}.

For \(\tau=0\), it follows that
\[
v=(\lambda I-\Delta)^{-1}(\mu u),
\]
and hence
\(
v\in C^{2+\alpha}([0,T]\times\R^N)
\)
uniformly in \(\varepsilon\); see also \cite[Theorem 6.2]{gilbarg1998elliptic}.

Now consider the case \(\tau=1\). If \(v_0\in C^{2+\alpha}(\R^N)\), then by \cite[Theorem 5.1, Chapter IV]{ladyzhenskaya1968linear},
\[
v\in C^{1+\alpha/2,\,2+\alpha}([0,T]\times\R^N).
\]
 This completes the proof of the the first part of Theorem 1.3(1).
\end{proof}

Next we establish uniqueness for weak solutions that are H\"older continous up to initial, for $1<m\le 3$. . Our argument is based on a duality method similar to that used in \cite[Theorem 1.1]{hassan2025global}.

\subsection{Uniqueness/Proof of Theorem 1.3(2)}
Let \((u_1,v_1)\) and \((u_2,v_2)\) be two bounded nonnegative weak solutions of \eqref{main-eq} with the same initial data
\(u_0\in C^\alpha(\mathbb R^N)\) and \(\tau v_0\in C^{2+\alpha}(\mathbb R^N)\). Assume that \(u_i\) is uniformly H\"older continuous in \(\Omega_T\) and that \(v_i\) is uniformly \(C^2\) in \(\Omega_T\), for \(i=1,2\). By Theorems~1.2 and~1.3(1), such weak solutions exist whenever \(b>b_m\). We define
\beq\lb{4.u}
\bar u := u_2-u_1,\qquad \bar v := v_2-v_1,\qquad \tilde u := u_2+u_1.
\eeq
Our goal is to show that \(\bar u=0\) and \(\bar v=0\). 
First, we need an estimate for the adjoint problem associated with the equation satisfied by \(\bar u\psi\). Let
\beq\lb{astar}
 a^*(t,x) :=
\left(\frac{u_2^m-u_1^m}{u_2-u_1}\right)(t,x).
\eeq
Consider extending the function $a^*(t,x)+\delta$ to $t\in [0,\infty)$ by
$$
a^{\delta,*}(t,x):=
a^*( \max\{0,\min\{T^*,t\}\},x)+\delta.
$$
Thus, it is clear that
\[
\int_0^{T^*} \int_{\R^N} |a^*-a^{\delta,*}|\bar u^2 \varphi^2 \le C \delta^2,
\]
where
\(
C:=\int_0^{T^*}\bar u^2\psi^2. 
\) Let $\beta\in C_c^\infty(\R^{N+1})$ be non-negative with compact support and unit total mass, and set
$$
a^{\delta,\epsilon}(t,x):=\frac{1}{\epsilon^{N+1}} \int_{\R}\int_{\R^N} \beta\Big(\frac{1}{\epsilon} (t-s),\frac{1}{\epsilon} (x-y)\Big )a^{\delta,*}(s,y)\,dy\, ds.
$$
Since $a^{\delta,*}$ is uniformly continuous on $\R\times\R^N$, 
\(
\lim_{\epsilon\to 0} a^{\delta,\epsilon}=a^{\delta,*}
\) uniformly on $\R\times\R^N$.
Hence, after taking $\epsilon>0$ sufficiently small,  
$a^\delta(t,x):=a^{\delta,\epsilon}$  is a smooth approximation of \(a^*\) satisfying
\beq\lb{4.K}
\max\{\delta, a^*\} \le a^\delta \le \max_{i=1,2}\left\{m(\|u_i\|_\infty+1)^{m-1}\right\}=:K \quad \text{and}\quad
 \int_0^{T^*} \int_{\R^N}|a^*-a^\delta|^2 \bar u^2 \psi^2\le C\delta^2.
\eeq
Next, for some \(T>0\) and nonnegative \(\xi(t,x)\) as above, consider the following adjoint problem:
\begin{equation}\label{varphi-eq}
\begin{cases}
\varphi_t + a^\delta \Delta\varphi + g^\delta\cdot\nabla\varphi + f^\delta\varphi + \xi = 0,\\
\varphi(T,x)=0,
\end{cases}
\end{equation}
where
\begin{equation}\label{gfeqn}
    g^\delta := 2a^\delta \frac{\nabla\psi}{\psi} + \chi\nabla v_2,
    \qquad
    f^\delta := a-b\tilde u + a^\delta \frac{\Delta \psi}{\psi}+ \chi \nabla v_2\cdot \frac{\nabla\psi}{\psi}.
\end{equation}
By the regularity of the solutions and the smoothness of \(\psi\), both \(g^\delta\) and \(f^\delta\) are bounded on \([0,T]\times\R^N\). The following estimate for the solution of the adjoint problem \eqref{varphi-eq} follows from Lemma 6.1 in \cite{hassan2025global}; the proof relies only on the boundedness of the functions \(u_i\) and the finiteness of \(\|v_2\|_{C^2}\) which holds by our assumption.

\begin{lem}\label{varphilem}
Fix any \(\kappa\in (0,\frac12)\). Assume that either \(a,b>0\) and \(1<m\le 3\) or $b=0$ and $m>1$.
For any \(\eta \in (0,\tfrac18)\), there exists \(T^*=T^*(\eta)>0\) such that, for every \(T\in(0,T^*]\), if \(\varphi=\varphi^{\delta,\xi}\) is the solution of \eqref{varphi-eq} on \([0,T]\times\R^N\), then
\begin{align}\label{gradvarphi}
&\int_t^T\!\!\int_{\mathbb{R}^N} a^\delta |\Delta\varphi^{\delta,\xi}|^2
 + \int_t^T\!\!\int_{\mathbb{R}^N} |\nabla\varphi^{\delta,\xi}|^2
 + \int_t^T\!\!\int_{\mathbb{R}^N} |\varphi^{\delta,\xi}|^2 \le \eta \Big(
      \int_t^T\!\!\int_{\mathbb{R}^N} |\xi|^2
    + \int_t^T\!\!\int_{\mathbb{R}^N} |\nabla\xi|^2
    \Big)
\end{align}
for every \(t\in[0,T]\).
\end{lem}

We now prove Theorem 1.3(2).

\begin{proof}[Proof of Theorem 1.3 (2)]
First, we prove that \eqref{main-u-estimate} holds on \([0,T^*]\). Let \(\varphi=\varphi^{\delta,\xi}\) be the solution of \eqref{varphi-eq} corresponding to some \(\kappa\in(0,\frac12)\), \(\delta\in(0,1)\), and \(T\in(0,T^*]\). Using \(\varphi\psi\) as a test function in the equations for \(u_i\), \(i=1,2\), and then taking the difference, we obtain
\begin{align}\label{ubarvarphi}
-\int_0^T \int_{\mathbb{R}^N} \bar u \,\varphi_t \psi
&= -\int_0^T \int_{\mathbb{R}^N} \nabla (\varphi \psi)\cdot \nabla (u_2^m-u_1^m)
+ \int_0^T \int_{\mathbb{R}^N}\chi u_2\nabla (\varphi \psi)\cdot \nabla v_2 \nonumber\\
&\qquad- \int_0^T \int_{\mathbb{R}^N} \chi u_1\nabla (\varphi \psi)\cdot \nabla v_1
+ \int_0^T \int_{\mathbb{R}^N}a\bar u \,\varphi\psi
- b(u_2^2-u_1^2)\varphi\psi \nonumber\\
&= \int_0^T \int_{\mathbb{R}^N}
\left[
\Delta (\varphi\psi)\, a^*\bar u
+ \chi\bar u \nabla v_2\cdot\nabla (\varphi \psi)
+\chi u_1\nabla\bar v\cdot\nabla(\varphi\psi)
+ (a-b\tilde u)\bar u\,\varphi\psi
\right].
\end{align}
On the other hand, multiplying \eqref{varphi-eq} by \(\psi\), we obtain
\[
0 = \varphi_t\psi + a^\delta \Delta(\varphi \psi) + \chi\nabla v_2\cdot\nabla(\varphi\psi)
+ (a-b\tilde u)\psi\varphi + \xi\psi .
\]
Multiplying this identity by \(\bar u\) and integrating over space and time yields
\begin{align}\label{varphiubar}
0
&=
\int_0^T \int_{\mathbb{R}^N} \bar u \,\varphi_t \psi
+ \int_0^T \int_{\mathbb{R}^N}
\left[
\Delta (\varphi\psi)\, a^\delta\bar u
+ \chi\bar u \nabla v_2\cdot\nabla (\varphi \psi)
+ (a-b\tilde u)\bar u\,\varphi \psi
+ \bar u\xi\psi
\right].
\end{align}
Subtracting \eqref{ubarvarphi} from \eqref{varphiubar}, and using H\"older's inequality together with \eqref{4.K}, we obtain
\begin{align*}
    \int_0^T \int_{\mathbb{R}^N} \bar u \,\xi\psi
    &= \int_0^T\int_{\R^N} (a^* - a^\delta)\bar u\,\Delta(\varphi\psi)
     + \int_0^T\int_{\R^N} \chi u_1\nabla\bar v\cdot\nabla(\varphi\psi) \\
    &\le \left(\int_0^T \int_{\mathbb{R}^N} \frac{|a^* - a^\delta|^2}{a^\delta}|\bar u|^2\psi^2 \right)^{1/2}
    \left(\int_0^T \int_{\mathbb{R}^N}\frac{a^\delta}{\psi^2}|\Delta (\varphi\psi)|^2\right)^{1/2} \\
    &\quad+ |\chi|\|u_1\|_\infty
    \left(\int_0^T\int_{\R^N} |\nabla\bar v|^2\psi^2\right)^{1/2}
    \left(\int_0^T\int_{\R^N}|\nabla(\varphi\psi)|^2\frac{1}{\psi^2}\right)^{1/2} \\
    &\le \delta^{-1/2}
    \left(\int_0^T \int_{\mathbb{R}^N}|a^* - a^\delta|^2|\bar u|^2\psi^2 \right)^{1/2}
    \left(\int_0^T \int_{\mathbb{R}^N}\frac{a^\delta}{\psi^2}|\Delta (\varphi\psi)|^2\right)^{1/2} \\
    &\quad+ |\chi|\|u_1\|_\infty
    \left(\int_0^T\int_{\R^N} |\nabla\bar v|^2\psi^2\right)^{1/2}
    \left(\int_0^T\int_{\R^N}|\nabla(\varphi\psi)|^2\frac{1}{\psi^2}\right)^{1/2}.
\end{align*}
By Lemma \ref{varphilem} and \eqref{4.K}, it follows that for some constant \(C_1\), depending only on \(\chi\) and \(\|u_1\|_\infty\),
\begin{align}
\label{baru-xi}
    \int_0^T \int_{\mathbb{R}^N} \bar u \,\xi\psi
    &\le (C\delta)^{1/2}\eta^{1/2}
    \left(\int_0^T \int_{\mathbb{R}^N}\big(|\xi|^2+|\nabla \xi|^2\big)\right)^{1/2} \nonumber\\
    &\quad + C_1\eta^{1/2}
    \left(\int_0^T\int_{\R^N} |\nabla\bar v|^2\psi^2\right)^{1/2}
    \left(\int_0^T\int_{\R^N} \big(|\xi|^2+|\nabla \xi|^2\big)\right)^{1/2}.
\end{align}
Letting \(\delta\to 0\), we obtain 
\begin{equation}\label{main-u-estimate}
    \int_0^T \int_{\mathbb{R}^N} \bar u \,\xi \psi
    \lesssim
    \left(\int_0^T\int_{\R^N} |\nabla\bar v|^2\psi^2\right)^{\frac{1}{2}}
    \left(\int_0^T\int_{\R^N}\big(|\xi|^2+|\nabla \xi|^2\big)\right)^{\frac{1}{2}}.
\end{equation}

\medskip

Next, we will use \eqref{main-u-estimate} to prove $\bar v= 0$, by choosing a suitable \(\xi\) in \eqref{main-u-estimate}. Since \(\bar v \psi\) and \(\nabla(\bar v\psi)\) decay exponentially in space, there exists a sequence \(\{\xi_n\}\) of smooth, compactly supported functions such that
\[
\int_0^T \int_{\R^N} |\xi_n-\bar v \psi|^2 +|\nabla \xi_n -\nabla (\bar v\psi)|^2 \to 0
\qquad\text{as }n\to\infty.
\]
Hence \eqref{main-u-estimate} holds with \(\xi=\mu \bar v \psi\). This implies
\begin{align}\label{barubarv}
\mu\int_0^{T} \int_{\mathbb{R}^N} \bar u \bar v \psi^2
 &\le C\eta^{1/2}
 \left(\int_0^{T}\int_{\R^N} |\nabla\bar v|^2\psi^2\right)^{1/2}
 \left(\int_0^{T}\int_{\R^N}|\bar v\psi|^2+|\nabla (\bar v\psi)|^2\right)^{1/2} \\
&\le \frac{1}{4} \int_0^{T}\int_{\R^N} |\nabla\bar v|^2\psi^2
 + C\eta\int_0^{T}\int_{\R^N}|\bar v\psi|^2 + |\nabla(\bar v \psi)|^2.
\end{align}

Now consider the equations satisfied by \(v_1\) and \(v_2\), and take their difference.

When \(\tau=1\), we have
\[
\bar v_t = \Delta \bar v - \lambda\bar v +\mu \bar u,
\qquad \bar v(0,x)=0.
\]
Multiplying this equation by \(\bar v\psi^2\) and integrating over space and time yields, for any \(T\in[0,T^*]\),
\begin{align}\label{barv}
  &\frac{1}{2} \int_{\R^N} \bar v^2(T,x)\psi^2(x)\,dx
  + \int_0^{T}\int_{\R^N} |\nabla \bar v|^2\psi^2
  + \lambda\int_0^{T}\int_{\R^N}\bar v^2\psi^2 \nonumber\\
&\le \int_0^{T}\int_{\R^N} \bar v|\nabla \bar v||\nabla(\psi^2)|
 + \mu \int_0^{T}\int_{\R^N}\bar u \bar v \psi^2 \nonumber\\
&\le \kappa\int_0^{T}\int_{\R^N} |\nabla \bar v|^2\psi^2
 + \mu \int_0^{T}\int_{\R^N}\bar u \bar v \psi^2
 + \kappa\int_0^{T}\int_{\R^N}\bar v^2\psi^2.
\end{align}

When \(\tau=0\), we have
\[
0 = \Delta \bar v - \lambda\bar v +\mu \bar u.
\]
Multiplying this equation by \(\bar v\psi^2\) and integrating over space and time yields, for any \(T\in[0,T^*]\),
\begin{align}\label{barv2}
  &\int_0^{T}\int_{\R^N} |\nabla \bar v|^2\psi^2
  + \lambda\int_0^{T}\int_{\R^N}\bar v^2\psi^2 \nonumber\\
&\le \int_0^{T}\int_{\R^N} \bar v|\nabla \bar v||\nabla(\psi^2)|
 + \mu \int_0^{T}\int_{\R^N}\bar u \bar v \psi^2 \nonumber\\
&\le \kappa\int_0^{T}\int_{\R^N} |\nabla \bar v|^2\psi^2
 + \mu \int_0^{T}\int_{\R^N}\bar u \bar v \psi^2
 + \kappa\int_0^{T}\int_{\R^N}\bar v^2\psi^2.
\end{align}

Thus, in both cases, \eqref{barv} and \eqref{barv2} yield
\[
\int_0^{T}\int_{\R^N} |\nabla \bar v|^2\psi^2
+ \lambda\int_0^{T}\int_{\R^N}\bar v^2\psi^2
\le
\kappa\int_0^{T}\int_{\R^N} |\nabla \bar v|^2\psi^2
+ \mu \int_0^{T}\int_{\R^N}\bar u \bar v \psi^2
+ \kappa\int_0^{T}\int_{\R^N}\bar v^2\psi^2.
\]
Using \eqref{barubarv}, we obtain, for \(\eta<\frac{\kappa}{4C}\),
\begin{align*}
  &\int_0^{T}\int_{\R^N} |\nabla \bar v|^2\psi^2
  + \lambda\int_0^{T}\int_{\R^N}\bar v^2\psi^2 \\
  &\le (2\kappa +\tfrac14)\int_0^{T}\int_{\R^N} |\nabla \bar v|^2\psi^2
  + 3\kappa\int_0^{T}\int_{\R^N}\bar v^2\psi^2.
\end{align*}
Choosing
\(
\kappa<\min\left\{\frac18,\frac{\lambda}{6}\right\},
\)
it follows that \(\bar v(t,\cdot)=0\) for every \(t\in[0,T^*]\). Then \eqref{main-u-estimate} implies that \(\bar u=0\) on \([0,T^*]\). Therefore,
\[
u_1=u_2
\qquad\text{and}\qquad
v_1=v_2
\quad\text{on }[0,T^*].
\]
Uniqueness on any interval \([0,\infty)\) then follows by iteration.
\end{proof}

\bibliographystyle{plain}
\bibliography{reference}

\end{document}